\documentclass[oneside,english]{amsart}
\usepackage[T1]{fontenc}
\usepackage[latin9]{inputenc}
\usepackage{mathrsfs}
\usepackage{amsbsy}
\usepackage{amstext}
\usepackage{amsthm}
\usepackage{amssymb}

\makeatletter
\numberwithin{equation}{section}
\numberwithin{figure}{section}

\usepackage{amsfonts}
\usepackage{amsmath}
\usepackage{amsfonts}\usepackage{amsthm}\usepackage{color}\usepackage{enumitem}\usepackage{bm}\usepackage{cancel}\usepackage{thmtools}
\def\theenumi{\arabic{enumi}}

\def\theenumii{\alph{enumii}}
\def\p@enumii{\theenumi.}

\def\theenumiii{\arabic{enumiii}}
\def\p@enumiii{(\theenumi)(\theenumii)}

\def\p@enumiv{\p@enumiii.\theenumiii}
\newtheorem{theorem}{Theorem}[section]

\newtheorem{claim}[theorem]{Claim}

\newtheorem{corollary}[theorem]{Corollary}

\newtheorem{fact}[theorem]{Fact}
\newtheorem{lemma}[theorem]{Lemma}

\newtheorem{proposition}[theorem]{Proposition}

\theoremstyle{definition}

\makeatother

\usepackage{babel}
\begin{document}
\title{Neretin groups admit no non-trivial invariant random subgroups}
\author{Tianyi Zheng}
\begin{abstract}
We show that Neretin groups have no non-trivial invariant random subgroups.
These groups provide first examples of non-discrete, compactly generated,
locally compact groups with this property.
\end{abstract}

\date{May 18, 2019}
\address{Tianyi Zheng, Department of Mathematics, UC San Diego, 9500 Gilman
Dr, La Jolla, CA 92093, USA. Email address: tzheng2@math.ucsd.edu.}
\maketitle

\section{Introduction}

Let $G$ be a locally compact group and denote by ${\rm Sub}(G)$
the space of closed subgroups of $G$ equipped with the Chabauty topology.
An \emph{invariant random subgroup} (IRS) of $G$, defined in \cite{AGV},
is a Borel probability measure on ${\rm Sub}(G)$ which is invariant
under conjugation by $G$. 

Normal subgroups corresponds to $\delta$-measures on ${\rm Sub}(G)$.
A subgroup $H$ of $G$ is said to be of finite co-volume, or \emph{co-finite}
for short, if $H$ is closed and $G/H$ carries a $G$-invariant probability
measure. A \emph{lattice} in $G$ is a discrete subgroup of finite
co-volume. Co-finite subgroups give rise to IRSs: the pushforward
of a $G$-invariant probability measure on $G/H$ under the map $gH\mapsto gHg^{-1}$
is an IRS. Thus IRSs can be viewed as generalizations of both normal
subgroups and lattices. It is natural to ask what properties of normal
subgroups or lattices can be extended to IRSs. In the other direction,
viewing lattices as elements in the space of IRSs on ${\rm Sub}(G)$
turns out to be a powerful tool in studying lattices, see \cite{7S,Gelander2}. 

Invariant random subgroups are closely related to probability measure
preserving (p.m.p.\@) actions. Given a p.m.p.\@ action $G\curvearrowright(X,m)$,
the pushforward of the probability measure $m$ under the stabilizer
map $x\mapsto{\rm St}_{G}(x)$ gives rise to an IRS, which we refer
to as the stabilizer IRS of the action $G\curvearrowright(X,m)$.
It is known that all IRSs arise in this way (\cite{AGV,7S}), and
moreover, an ergodic IRS arises as the stabilizer IRS of an ergodic
p.m.p.\@ action (\cite[Proposition 3.5]{CreutzPeterson}). 

We say that $G$ has no non-trivial IRSs if every IRS is a convex
combination of $\delta_{\{id\}}$ and $\delta_{G}$. By the characterization
of IRSs in terms of stabilizers as cited above, $G$ has no non-trivial
IRSs if and only if every non-trivial ergodic p.m.p.\@ action of
$G$ is essentially free. Recall that an action is \emph{essentially
free }if there is a full measure subsets consisting of points with
trivial stabilizer.

In \cite{7S2} it is asked whether there exists a simple, non-discrete
locally compact group which does not have non-trivial IRSs; and Neretin
groups are proposed as candidates. A more detailed discussion of this
question can be found in the survey \cite{Gelander1}. The supporting
evidences are that Neretin groups are abstractly simple by \cite{kapoudjian},
and they are first examples of locally compact group which do not
admit any lattices by \cite{BCGM}. Note that many examples of groups
with no nontrivial IRSs can be found among countable groups, see for
example \cite{Dudko-Medynets2}. First examples of non-discrete locally
compact groups with no nontrivial IRSs are constructed in \cite{LBMB2}.
The groups constructed in \cite{LBMB2} are not compactly generated. 

Let $\mathcal{T}$ be a $(d+1)$-regular unrooted tree. The Neretin
group $\mathcal{N}_{d}$ is the group of \emph{almost automorphisms}
of $\mathcal{T}$, or equivalently, the group of \emph{spheromorphisms}
of $\partial\mathcal{T}.$ The group $\mathcal{N}_{d}$ is introduced
by Neretin in \cite{neretin} as combinatorial analogues of the group
of diffeomorphisms of the circle, with some ideas tracing back to
his earlier work \cite{neretin1}. There is a unique group topology
on $\mathcal{N}_{d}$ such that the natural inclusion ${\rm Aut}(\mathcal{T})\hookrightarrow\mathcal{N}_{d}$
is continuous and open, and endowed with this topology, $\mathcal{N}_{d}$
is locally compact and compactly generated, see \cite{caprace-de-dedts}.
Neretin groups are now fundamental examples in the growing structure
theory of totally disconnected locally compact groups, see \cite{CRW1,CRW2}
and references therein. 

The main goal of the present work is to show that Neretin groups admit
no nontrivial IRSs. Our argument applies to a generalization of Neretin
groups, called coloured Neretin groups, which are introduced and studied
recently in \cite{Lederle}. We now briefly describe these groups,
more precise definitions are recalled in Section \ref{sec:Preliminaries}.
For every vertex of $\mathcal{T}$, fix a bijection from the edges
incident to it to the set of colours $D=\{0,1,\ldots,d\}$. Given
a subgroup $F\le{\rm Sym}(D)$, Burger and Mozes \cite{Burger-Mozes}
constructed a closed subgroup of ${\rm Aut}(\mathcal{T})$, denoted
by $U(F)$, which is the universal group with local actions at every
vertex in $F$. The \emph{coloured Neretin group }$\mathcal{N}_{F}$
is defined as the group of $U(F)$-almost automorphisms. It is shown
in \cite{Lederle} that there is a unique group topology on $\mathcal{N}_{F}$
such that the inclusion $U(F)\hookrightarrow\mathcal{N}_{F}$ is continuous
and open, and endowed with this topology, $\mathcal{N}_{F}$ is locally
compact and compactly generated. 

\begin{theorem}\label{main}

Let $F<{\rm Sym}(D)$ be any subgroup. Let $\mu$ be an ergodic IRS
of the coloured Neretin group $\mathcal{N}_{F}$. Then either $\mu=\delta_{\{id\}}$
or $\mu$-a.e. $H$ contains the derived subgroup $\mathcal{N}_{F}'=\left[\mathcal{N}_{F},\mathcal{N}_{F}\right]$.
In particular, the Neretin group $\mathcal{N}_{d}$, corresponding
to the case $F={\rm Sym}(D)$, admits no non-trivial IRSs. 

\end{theorem}

For general $F$, by \cite{Lederle} $\mathcal{N}_{F}'$ is simple,
open and of finite index in $\mathcal{N}_{F}$. In particular, Theorem
\ref{main} implies that $\mathcal{N}_{F}$ has no lattices, answering
\cite[Question 1.5]{Lederle} by removing the constraints on $F$. 

The key step in the proof of Theorem \ref{main} is the following
statement on containment of rigid stabilizers. Let $G$ be a group
acting on a topological space $X$ by homeomorphisms and $U\subseteq X$
an open subset. Denote by $R_{G}(U)$ the \emph{rigid stabilizer}
of $U$ in $G$, that is, $R_{G}(U)=\{g\in G:x\cdot g=x\mbox{ for all }x\in X\setminus U\}$.
Given a finite subtree $A$ of $\mathcal{T},$ denote by $B_{n}(A)$
the subtree with vertices within distance $n$ to $A$. Denote by
$O_{F}^{A}$ the subgroup which consists of almost automorphisms which
can be represented by a triple $\left(B_{n}(A),B_{n}(A),\varphi\right)$
for some $n\in\mathbb{N}$, see the precise definition in Section
\ref{sec:Induced}. When $A$ consists of a single vertex, $O_{F}^{A}$
is the same as the group $O$ considered in \cite{BCGM,Lederle}.
Note that $O_{F}^{A}$ is an open subgroup of $\mathcal{N}_{F}$. 

\begin{proposition}\label{contain}

Let $A$ be a finite complete subtree of $\mathcal{T}$ and $\mu_{A}$
be an IRS of $O_{F}^{A}$. Then $\mu_{A}$-a.e. $H$ satisfies the
following: if $H\neq\{id\}$, then there exists a non-empty open set
$U\subseteq\partial\mathcal{T}$ such that 
\[
\left[R_{O_{F}^{A}}(U),R_{O_{F}^{A}}(U)\right]<H.
\]

\end{proposition}

The statement of Proposition \ref{contain} is an exact analogue of
the double commutator lemma for IRSs of a countable group in \cite[Theorem 1.2]{rigidstab}.
In the countable setting, one can show the double commutator lemma
for IRSs with rather soft arguments. It seems to be an interesting
question to what extent such a result is true in the non-discrete
setting. Proposition \ref{contain} is proved via studying induced
IRSs of finite sub-quotients of $O_{F}^{A}$, see more discussion
below. 

The Higman-Thompson group $V_{d,d+1}$ embeds in $\mathcal{N}_{d}$
as a dense subgroup, see \cite{caprace-de-dedts}. It is observed
in \cite{Nek-Cuntz} that the topological full group of the one-sided
Bernoulli shift over the alphabet with $d$ letters is isomorphic
to a Higman-Thompson group. More generally, topological full groups
of one-sided irreducible shifts of finite type are introduced and
investigated in \cite{Matui2015}. For the coloured Neretin group
$\mathcal{N}_{F}$, it is shown in \cite[Theorem 3.9]{Lederle} that
$\mathcal{N}_{F}$ has a dense subgroup $V_{F}$, which can be identified
as the topological full group of a one-sided irreducible shift of
finite type. By \cite[Corollary 3.9]{Dudko-Medynets2}, the countable
group $V_{F}'$ does not have non-trivial IRSs. 

Proposition \ref{contain} allows us to transfer the problem of IRSs
of $\mathcal{N}_{F}$ to $V_{F}$ by considering the intersection
map $H\mapsto H\cap V_{F}$. More precisely, Proposition \ref{contain}
guarantees that almost surely $H\cap V_{F}\neq\{id\}$, so that known
results on IRSs of $V_{F}$ as cited above can be applied, see Section
\ref{contain}.

Most of this paper is devoted to the proof of Proposition \ref{contain}.
The basic idea in the proof is that in the finite sub-quotients of
$O_{F}^{A}$ considered, if a subgroup does not contain a large finite
alternating group, then the probability that a random conjugate of
it containing a specific kind of almost automorphisms is small, quantitatively.
Then the Borel-Cantelli lemma can be applied to combine the estimates
in finite sub-quotients to obtain almost sure statements on the IRS.
In Section \ref{sec:Effective} we formulate two general bounds for
IRSs in countable groups in terms of subgroup index (Lemma \ref{E1})
and conjuagcy class size (Lemma \ref{E2}). The starting point of
the proof in \cite{BCGM} for absence of lattices in $O$ is a co-volume
estimate in the finite sub-quotients, which is later confronted by
the discreteness of the lattice. In some sense the subgroup index
Lemma \ref{E1} is a replacement for co-volume bounds in the context
of IRSs, although it is weaker. An outline of the proof of Proposition
\ref{contain} can be found in Section \ref{sec:Induced} after introducing
the necessary objects. We mention that the proof is rather self-contained:
the only result on finite symmetric groups invoked is the Praeger-Saxl
bound \cite{prager-saxl} on the orders of primitive subgroups. 

Following \cite{CRW1,CRW2}, let $\mathscr{S}$ be the class of all
non-discrete, compactly generated, locally compact groups that are
topologically simple. There is an evolving theory which treats $\mathscr{S}$
as a whole, see the survey \cite{Caprace} and references therein.
The class $\mathscr{S}$ naturally divides into two subclasses, $\mathscr{S}_{{\rm Lie}}$
which consists of connected Lie groups in $\mathscr{S}$; and $\mathscr{S}_{{\rm td}}$
which consists of totally disconnected groups in $\mathscr{S}$. Motivated
by the theory of lattices in semisimple Lie groups, it is natural
to investigate lattices and more generally, IRSs of groups in the
class $\mathscr{S}_{{\rm td}}$. It is reasonable to expect that an
abundance of examples of non-discrete compactly generated locally
compact groups with no non-trivial IRSs can be found in the class
$\mathscr{S}_{{\rm td}}$: for instances, some topological full groups
similar to Neretin type groups and certain simple groups acting on
trees with almost prescribed local action introduced and studied in
\cite{LB}. Our proof relies on properties of finite symmetric groups.
It is interesting to develop a more conceptual and robust approach
that could contribute to the study of $\mathscr{S_{{\rm td}}}$. 

\subsection*{Organization of the paper}

In Section \ref{sec:Effective} we formulate two quantitative bounds
for IRSs of countable groups, which might be useful as general tools.
Section \ref{sec:Preliminaries} contains preliminaries on Neretin
type groups. In Section \ref{sec:Induced} the induced IRSs in certain
sub-quotients and relevant events we consider are introduced. In Section
\ref{sec:Proof-of-Theorem} we explain how to deduce Theorem \ref{main}
from Proposition \ref{contain}. Section \ref{sec:Tree-match} contains
an auxiliary bound for the probability of two randomly chosen sets
to be in the same orbit of some tree automorphism group. In Section
\ref{sec:transitive} we present the proof of Proposition \ref{contain}
when $F$ is transitive. At the end of Section \ref{sec:transitive},
the proof of Theorem \ref{main} in the case of transitive $F$, e.g.
$F={\rm Sym}(D)$, is complete. Section \ref{sec:colored} explains
the additional arguments needed to prove Proposition \ref{contain}
for general $F$. 

\subsection*{Acknowledgment}

We are grateful to Nichol\'as Matte Bon for sharing his insights
on Neretin groups. We thank Mikl\'os Ab\'ert for helpful discussions. 

\section{Two counting lemmas for IRSs of countable groups\label{sec:Effective}}

Let $G$ be a locally compact second countable group and ${\rm Sub}(G)$
the space of closed subgroups of $G$. Recall that a pre-basis of
open sets for the \emph{Chabauty topology} is given by sets of the
form 
\[
\left\{ H\in{\rm Sub}(G):H\cap V\neq\emptyset\right\} ,\ \left\{ H\in{\rm Sub}(G):\ H\cap K=\emptyset\right\} ,
\]
where $V$ is a relatively compact open subset of $G$, and $K$ a
compact subset of $G$. The space ${\rm Sub}(G)$ endowed with the
Chabauty topology is compact and metrizable. 

In this section we formulate two quantitative bounds which exploit
the conjugation invariance of an IRS $\mu$. These bounds are applicable
to general countable groups, finite or infinite. 

The first lemma bounds the probability that a random subgroup with
distribution $\mu$ intersects a given set, in terms of the size of
the set and certain subgroup index. In order to state the bound we
introduce some notations. Let $\Gamma\curvearrowright X$ by homeomorphisms
and $U,V$ be two disjoint non-empty open sets in $X$. Given a subgroup
$H$ of $\Gamma$, let
\begin{equation}
H_{U\to V}:=\left\{ h\in H:\ V=U\cdot h\right\} ,\label{eq:SHU}
\end{equation}
and
\begin{equation}
\bar{H}_{U\to V}:=\left\{ h|_{U}:\ h\in H_{U\to V}\right\} .\label{eq:BSHU}
\end{equation}
Elements of $\bar{H}_{U\to V}$ are viewed as partial homeomorphisms
with domain $U$ and range $V$, denoted by $h|_{U}:U\to V$. Let
$\Omega_{U,V}$ be the event that $H$ contains an element which maps
$U$ to $V$, that is,
\[
\Omega_{U,V}=\left\{ H\in{\rm Sub}(\Gamma):\ H_{U\to V}\neq\emptyset\right\} .
\]

Recall that $R_{\Gamma}(U)$ denotes the rigid stabilizer of $U$
in $\Gamma$. In probabilistic expressions involving $\mathbb{E}_{\mu}$
or $\mathbb{P}_{\mu}$, the symbol $H$ denotes a random subgroup
with distribution $\mu$. Write ${\bf 1}_{\Omega}$ for the indicator
of the set $\Omega.$

\begin{lemma}[Subgroup index Lemma]\label{E1}

Let $\Gamma$ be a countable group acting faithfully on a topological
space $X$ by homeomorphisms. Let $U,V$ be two disjoint non-empty
open sets such that $\Gamma_{U\to V}\neq\emptyset$. Let $\mu$ be
an IRS of $\Gamma$. Then for any finite subset $A$ of partial homeomorphisms
in $\bar{\Gamma}_{U\to V}$, we have that $\mu$-a.e. $H$ with $\bar{H}_{U\to V}\cap A\neq\emptyset$
satisfies $\left|R_{\Gamma}(U):\left(\bar{H}_{U\to U}\cap R_{\Gamma}\left(U\right)\right)\right|<\infty$.
Moreover,
\begin{multline}
\mathbb{P}_{\mu}\left(\bar{H}_{U\to V}\cap A\neq\emptyset\right)\\
\le\mathbb{E}_{\mu}\left[\min\left\{ \frac{|A|}{\left|R_{\Gamma}(U):\left(\bar{H}_{U\to U}\cap R_{\Gamma}\left(U\right)\right)\right|},1\right\} {\bf 1}_{\Omega_{U,V}}(H)\right].\label{eq:index}
\end{multline}

\end{lemma}

In the statement of the previous lemma it is understood that in the
expression $\left|R_{\Gamma}(U):\left(\bar{H}_{U\to U}\cap R_{\Gamma}\left(U\right)\right)\right|$,
both $R_{\Gamma}(U)$ and $\bar{H}_{U\to U}$ are viewed as groups
of homeomorphisms of $U$. The second lemma is in the setting of product
of two groups. It bounds the probability that a random subgroup contains
a given set of group elements $B$, in terms of the size of the conjugacy
class of some coset associated with $B$. Given a subset $B\subseteq\Gamma$
of a subgroup $W<\Gamma$, denote by ${\rm Cl}_{W}(B)$ the collection
$W$-conjugates of $B$, that is
\[
{\rm Cl}_{W}(B)=\left\{ g^{-1}Bg:\ g\in W\right\} .
\]

\begin{lemma}[Conjugacy class size Lemma]\label{E2}

Suppose $\Gamma$ is a subgroup of the product $L_{1}\times L_{2}$,
where $L_{1},L_{2}$ are countable. Denote by $\pi_{i}$ the projection
$L_{1}\times L_{2}\to L_{i}$, $i=1,2$. Let $\mu$ be an IRS of $\Gamma$.
Then for any subset $B\subseteq\Gamma$, we have that $\mu$-a.e.
$H$ with $H\supseteq B$ satisfies $\left|{\rm Cl}_{N_{1}}\left(\pi_{1}(B)H_{1}\right)\right|<\infty$.
Moreover, 
\begin{equation}
\mathbb{P}_{\mu}\left(H\supset B\right)\le\mathbb{E}_{\mu}\left[\frac{1}{\left|{\rm Cl}_{N_{1}}\left(\pi_{1}(B)H_{1}\right)\right|}\mathbf{1}_{\left\{ \pi_{2}(A)\subseteq\pi_{2}(H)\right\} }\right],\label{eq:conj}
\end{equation}
where $H_{1}=H\cap\left(L_{1}\times\left\{ id_{\Gamma_{2}}\right\} \right)$
and $N_{1}$ is the normalizer of $H_{1}$ in $\pi_{1}(\Gamma)$. 

\end{lemma}

Lemma \ref{E1} and \ref{E2} can be used in conjunction as follows.
Start with a pair of open sets $U,V$ with $U\cap V=\emptyset$ and
$P\subset\bar{\Gamma}_{U\to V}$ such that $\mu\left(\bar{H}_{U\to V}\cap P\neq\emptyset\right)>0$.
Then Lemma \ref{E1} provides information on $\bar{H}_{U\to U}$,
and moreover, those $H$ with large index $\left|R_{\Gamma}(U):\left(\bar{H}_{U\to U}\cap R_{\Gamma}\left(U\right)\right)\right|$
make small contribution to the probability $\mu\left(\bar{H}_{U\to V}\cap P\neq\emptyset\right)$.
Next consider the induced IRS in $\Gamma_{U\to U},$which is a subgroup
of the product $L_{1}\times L_{2}$, where $L_{1}=\pi_{U}\left(\Gamma_{U\to U}\right)$
and $L_{2}=\pi_{U^{c}}\left(\Gamma_{U\to U}\right)$. Then Lemma \ref{E2}
provides information on sizes of conjugacy classes in the quotient
group $\bar{H}_{U\to U}/R_{H}(U)$. Such information can be useful
towards showing that $R_{H}(U)$ must contain certain subgroups. 

Given a non-discrete t.d.l.c.\@ group, to apply such estimates towards
understanding its IRSs, one first needs to choose a collection of
finite sub-quotients and consider the induced IRSs. For Neretin groups,
unlike countable groups discussed in \cite{rigidstab}, Lemma \ref{E1}
and \ref{E2} applied to induced IRSs are useful, but far from being
sufficient to conclude containment of rigid stabilizers. We will need
additional probability estimates in the finite sub-quotients in the
next sections. Such estimates heavily depend on the properties of
finite symmetric groups. 

The rest of this section is devoted to the proof of the two lemmas.
We follow notations of regular conditional distributions in the book
\cite[Chapter V.8]{Parthasarathy}. Let $(X,\mathcal{B})$, $(Y,\mathcal{C})$
be two Borel spaces, $\mathbb{P}$ a probability measure on $\mathcal{B}$
and $\pi:X\to Y$ a measurable map. Let $\mathbb{Q}=\mathbb{P}\circ\pi^{-1}$
be probability measure on $\mathcal{C}$ which is the pushforward
of $\mathbb{P}$. A \emph{regular conditional distribution} given
$\pi$ is a mapping $y\mapsto\mathbb{P}(y,\cdot)$ such that 

(i) for each $y\in Y$, $\mathbb{P}(y,\cdot)$ is a probability measure
on $\mathcal{B}$;

(ii) there exists a set $N\in\mathcal{C}$ such that $\mathbb{Q}(N)=0$
and for each $y\in Y\setminus N$, $\mathbb{P}(y,X\setminus\pi^{-1}(\{y\}))=0$;

(iii) for any $A\in\mathcal{B}$, the map $y\mapsto\mathbb{P}(y,A)$
is $\mathcal{C}$-measurable and 
\[
\mathbb{P}(A)=\int_{Y}\mathbb{P}(y,A)d\mathbb{Q}(y).
\]
We will refer to these three items as properties (i),(ii),(iii) of
a regular conditional distribution. 

Recall that a measure space $(X,\mathcal{B})$ is called a standard
Borel space if it is isomorphic to some Polish space equipped with
the Borel $\sigma$-field. It is classical that if $(X,\mathcal{B})$
and $(Y,\mathcal{C})$ are standard Borel spaces and $\pi:X\to Y$
is measurable, then there exists such a regular conditional distribution
$y\mapsto\mathbb{P}(y,\cdot)$ with properties (i),(ii),(iii); and
moreover it is unique: if $\mathbb{P}'(y,\cdot)$ is another such
mapping, then $\{y:\ \mathbb{P}'(y,\cdot)\neq\mathbb{P}(y,\cdot)\}$
is a set of $\mathbb{Q}$-measure $0$, see \cite[Theorem 8.1]{Parthasarathy}. 

In the proofs below, the outline is the same as in \cite{rigidstab}.
We keep track of the subgroup index and conjugacy class sizes which
appear in the argument, which naturally lead to the bounds stated
in Lemma \ref{E1} and \ref{E2}. 

\begin{proof}[Proof of Lemma \ref{E1}]

For $U,V$ such that $\mu(\Omega_{U,V})=0$, the statement of the
lemma is trivially true. Take a pair of $U,V$ such that $\mu(\Omega_{U,V})>0$
and consider the random variables $H_{U\to V}$, $\bar{H}_{U\to V}$
and $\bar{H}_{U\to U}$ as defined in (\ref{eq:SHU}), (\ref{eq:BSHU}).
Denote by $\mathbb{P}_{U,V}^{\mu}(\bar{H}_{U\to U},\cdot)$ the regular
conditional distribution of $\left(\bar{H}_{U\to V},\bar{H}_{U\to U}\right)$
given $\bar{H}_{U\to U}$, where $H$ has distribution $\mu_{U,V}=\mu\left(\cdot|\Omega_{U\to V}\right)$
on $\Omega_{U\to V}$. 

Since $\bar{H}_{U\to V}$ is a coset of $\bar{H}_{U\to U}$ and $\Gamma$
is countable, we have that $\mathbb{P}_{U,V}^{\mu}(\bar{H}_{U\to U},\cdot)$
is a probability measure on a countable set. Conjugation invariance
of $\mu$ implies that any $g\in G_{U\to V}$ and $\gamma\in R_{\Gamma}(V)$,
\begin{equation}
\mathbb{P}_{U,V}^{\mu}\left(\bar{H}_{U\to U},\left\{ \left(\bar{H}_{U\to U}g|_{U},\bar{H}_{U\to U}\right)\right\} \right)=\mathbb{P}_{U,V}^{\mu}\left(\bar{H}_{U\to U},\left\{ \left(\bar{H}_{U\to U}g|_{U}\gamma|_{V},\bar{H}_{U\to U}\right)\right\} \right),\label{eq:invariant}
\end{equation}
see \cite[Lemma 2.3]{rigidstab}. For $\mu$-a.e. $H\in\Omega_{U,V}$,
there must exist a coset $\bar{H}_{U\to U}\sigma|_{U}$, $\sigma\in G_{U\to V}$
depending on $\bar{H}_{U\to U}$, such that $\mathbb{P}_{U,V}^{\mu}\left(\bar{H}_{U\to U},\left\{ \left(\bar{H}_{U\to U}\sigma|_{U},\bar{H}_{U\to U}\right)\right\} \right)>0$.
If the number of right cosets $\bar{H}_{U\to U}\sigma|_{U}\gamma|_{V}$,
where $\gamma$ is taken over elements of $R_{\Gamma}(V)$, is infinite,
then the probability measure $\mathbb{P}_{U,V}^{\mu}(\bar{H}_{U\to U},\cdot)$
cannot be invariant under right multiplication as in (\ref{eq:invariant}).
Therefore there are only finitely many cosets of $\bar{H}_{U\to U}\sigma|_{U}$
in this collection. Denote by $\ell\left(\bar{H}_{U\to U}\sigma|_{U}\right)$
the number of cosets 
\begin{equation}
\ell\left(\bar{H}_{U\to U}\sigma|_{U}\right)=\left|\left\{ \bar{H}_{U\to U}\sigma|_{U}\gamma|_{V},\gamma\in R_{\Gamma}(V)\right\} \right|.\label{eq:numell}
\end{equation}
In other words, there are $\ell=\ell\left(\bar{H}_{U\to U}\sigma|_{U}\right)$
representatives $\gamma_{1},\ldots,\gamma_{\ell}$ in $R_{\Gamma}(V)$
such that for any $\gamma\in R_{\Gamma}(V)$, we have $\bar{H}_{U\to U}\sigma|_{U}\gamma|_{V}=\bar{H}_{U\to U}\sigma|_{U}\gamma_{k}|_{V}$
for exactly one $k\in\{1,\ldots,\ell\}$. It follows that for any
$\gamma\in R_{\Gamma}(V)$, there is a representative $\gamma_{k}$,
$k\in\{1,\ldots,\ell\}$, such that $\bar{H}_{U\to U}$ contains $\sigma|_{U}(\gamma\gamma_{k}^{-1})|_{V}\sigma|_{U}^{-1}$.
Consider the subgroup $R_{1}$ of $R_{\Gamma}(V)$ generated by the
collection $\gamma\gamma_{k}^{-1}$, where $\gamma\in R_{\Gamma}(V)$
and $\gamma_{k}$ is its corresponding representative. It's clear
by definition of $R_{1}$ that $\cup_{j=1}^{\ell}R_{1}\gamma_{j}=R_{\Gamma}(V)$,
therefore $R_{1}$ is a subgroup of $R_{\Gamma}(V)$ with index at
most $\ell$. Recall that $\sigma$ maps $U$ to $V$, therefore $\sigma R_{\Gamma}(V)\sigma^{-1}=R_{\Gamma}(U)$.
Let $T_{1}=\sigma R_{1}\sigma^{-1}$, it is a subgroup of $R_{\Gamma}(U)$
of index at most $\ell$. Elements of $T_{1}$ satisfy the property
that $\bar{H}_{U\to U}=\bar{H}_{U\to U}\gamma|_{U}$, in other words,
$\pi_{U}(T_{1})\le\bar{H}_{U\to U}$. Note that we have bounds on
the index
\begin{align}
\left|R_{\Gamma}(U):\left(R_{\Gamma}(U)\cap\bar{H}_{U\to U}\right)\right| & \le\left|\pi_{U}\left(R_{\Gamma}(V)\right):\pi_{U}\left(T_{1}\right)\right|\nonumber \\
 & \le\left|R_{\Gamma}(V):T_{1}\right|\nonumber \\
 & \le\ell\left(\bar{H}_{U\to U}\sigma|_{U}\right).\label{eq:index2}
\end{align}
The first statement on finite index follows. Now we proceed to prove
(\ref{eq:index}). Take any $g\in\Gamma_{U\to V}$. Then by property
(iii) of regular conditional probability, we have 
\begin{multline*}
\mathbb{P}_{\mu}\left(\bar{H}_{U\to V}\cap A\neq\emptyset\right)\\
=\mu\left(\Omega_{U,V}\right)\mathbb{E}_{\mu_{U,V}}\left[\sum_{\left(\bar{H}_{U\to U}g|_{U}\right)\cap A\neq\emptyset}\mathbb{P}_{U,V}^{\mu}\left(\bar{H}_{U\to U},\left\{ \left(\bar{H}_{U\to U}g|_{U},\bar{H}_{U\to U}\right)\right\} \right)\right],
\end{multline*}
where the summation is over those cosets in $\left\{ \bar{H}_{U\to U}g|_{U}:g\in\Gamma_{U\to V}\right\} $
with non-empty intersection with $A$. By the same reasoning as in
the previous paragraph, translation invariance (\ref{eq:invariant})
implies that for each coset,
\begin{multline*}
\mathbb{P}_{U,V}^{\mu}\left(\bar{H}_{U\to U},\left\{ \left(\bar{H}_{U\to U}g|_{U},\bar{H}_{U\to U}\right)\right\} \right)\\
\le\frac{1}{\ell(\bar{H}_{U\to U}g|_{U})}\le\frac{1}{\left|R_{\Gamma}(U):\left(R_{\Gamma}(U)\cap\bar{H}_{U\to U}\right)\right|},
\end{multline*}
where $\ell(\bar{H}_{U\to U}g|_{U})$ is the number of cosets defined
in (\ref{eq:numell}) and in the last step we plugged in (\ref{eq:index2}).
Since the cosets are disjoint, there are at most $|A|$ of them that
intersect $A$. It follows that 
\begin{multline*}
\mathbb{P}_{\mu}\left(\bar{H}_{U\to V}\cap A\neq\emptyset\right)\\
\le\mu\left(\Omega_{U,V}\right)\mathbb{E}_{\mu_{U,V}}\left[{\bf 1}_{\left\{ \left|R_{\Gamma}(U):\left(R_{\Gamma}(U)\cap\bar{H}_{U\to U}\right)\right|\le|A|\right\} }+\frac{|A|{\bf \cdot1}_{\left|R_{\Gamma}(U):\left(R_{\Gamma}(U)\cap\bar{H}_{U\to U}\right)\right|>|A|}}{\left|R_{\Gamma}(U):\left(R_{\Gamma}(U)\cap\bar{H}_{U\to U}\right)\right|}\right].
\end{multline*}
The statement follows. 

\end{proof}

\begin{proof}[Proof of Lemma \ref{E2}]

Denote by $\text{\ensuremath{\mathcal{A}_{B}} the event \ensuremath{\{H\in{\rm Sub}(\Gamma):\ \pi_{2}(B)\subset\pi_{2}(H)\}}}.$
If $\mu\left(\mathcal{A}_{B}\right)=0$ then the statement is trivially
true. We may assume $\mu\left(\mathcal{A}_{B}\right)>0$. Recall that
for any $H<L_{1}\times L_{2}$, there is an isomorphism $\varphi_{H}:\pi_{2}(H)/H_{2}\to\pi_{1}(H)/H_{1}$,
given by the map $h_{2}\mapsto\{h_{1}\in\pi_{1}(H):(h_{1},h_{2})\in H\}$,
see \cite[Fact 3.1]{rigidstab}. We refer to $\varphi$ as the paring
in $H$ between the two coordinates. Denote by $\mathbb{P}_{B}^{\mu}$
the regular conditional distribution of $\left(\varphi_{H}(\pi_{2}(B)),H_{1}\right)$
given the random variable $H_{1}$, where $H$ has distribution $\mu\left(\cdot|\mathcal{A}_{B}\right)$
on $\mathcal{A}_{B}$. Then the conjugation invariance property of
$\mu$ implies that 
\[
\mathbb{P}_{B}^{\mu}\left(H_{1},\left\{ \left(\pi_{1}(B)H_{1},H_{1}\right)\right\} \right)=\mathbb{P}_{B}^{\mu}\left(H_{1},\left\{ \left(g^{-1}\pi_{1}(B)H_{1}g,H_{1}\right)\right\} \right),
\]
for any $g\in N_{1}$, see \cite[Lemma 3.3]{rigidstab}. It follows
that 
\[
\mathbb{P}_{B}^{\mu}\left(H_{1},\left\{ \left(\pi_{1}(B)H_{1},H_{1}\right)\right\} \right)\le\frac{1}{{\rm Cl}_{N_{1}}\left(\pi_{1}(B)H_{1}\right)}.
\]
In order for $B$ to be contained in $H$, it is necessarily that
$\pi_{2}(B)$ is paired with $\pi_{1}(B)$ under $\varphi_{H}$. Thus,
by property (iii) of regular conditional distribution, we have 
\begin{align*}
\mathbb{P}_{\mu}(B & \subset H)\le\mathbb{P}(H\in\mathcal{A}_{B}\mbox{ and }\varphi_{H}(\pi_{2}(B))=\pi_{1}(B)H_{1})\\
 & =\mu\left(\mathcal{A}_{B}\right)\mathbb{E}_{\mu(\cdot|\mathcal{A}_{B})}\left[\mathbb{P}_{B}^{\mu}\left(H_{1},\left\{ \left(\pi_{1}(B)H_{1},B\right)\right\} \right)\right]\\
 & \le\mathbb{E}_{\mu}\left[\frac{1}{\left|{\rm Cl}_{N_{1}}\left(\pi_{1}(B)H_{1}\right)\right|}\mathbf{1}_{\mathcal{A}_{B}}(H)\right].
\end{align*}

\end{proof}

\section{Preliminaries on Neretin-type groups\label{sec:Preliminaries}}

Terminologies and notations in this section follow \cite{Lederle}. 

Let $\mathcal{T}=\mathcal{T}_{d+1}$ be a (unrooted) regular tree
of degree $d+1$. Denote by ${\rm Aut}(\mathcal{T})$ the group of
automorphisms of $\mathcal{T}$, equipped with the topology of pointwise
convergence. We fix, once and for all, a reference point $v_{0}\in\mathcal{T}$,
and a legal colouring of (geometric) edges of $\mathcal{T}.$ Recall
that a legal colouring is a map ${\rm col}$ from (geometric) edges
of $\mathcal{T}$ to the set $D=\{0,1,\ldots,d\}$, such that at every
vertex the edges incident to it have different colours. 

Denote by $\partial\mathcal{T}$ the boundary of $\mathcal{T}$, which
consists of all infinite geodesic rays starting at $v_{0}$. Given
a vertex $v\in\mathcal{T},$denote by $C_{v}$ the subset of $\partial\mathcal{T}$
which consists of infinite geodesic rays that starts at $v_{0}$ and
passes through $v$. As usual, $\partial\mathcal{T}$ is equipped
with the topology generated by the basis $\left\{ C_{v}\right\} _{v\in\mathcal{T}}$. 

Let $A$ be a finite subtree of $\mathcal{T}.$ The subtree $A$ is
called \emph{complete} if it contains the reference point $v_{0}$
and if a vertex $v\in A$ is not a leaf, then all of its children
are contained in $A$. Denote by $\partial A$ the set of leaves of
$A$. By $\mathcal{T}\setminus A$ we mean the subgraph $\coprod_{v\in\partial A}\mathcal{T}_{v}$,
that is the disjoint union (forest) of subtrees rooted at leaves of
$A$. 

An \emph{almost automorphism} of $\mathcal{T}$ is represented by
a triple $(A,B,\varphi),$ where $A,B\subseteq\mathcal{T}$ are complete
finite subtrees such that $\left|\partial A\right|=\left|\partial B\right|$,
and $\varphi:\mathcal{T}\setminus A\to\mathcal{T}\setminus B$ is
a forest isomorphism. Two such triples are equivalent if up to enlarging
the subtrees $A,B$ they coincide. An almost automophism is the equivalence
class of such a representing triple. An almost automorphism of $\mathcal{T}$
induces a homeomorphism of $\partial\mathcal{T}$, called a \emph{spheromorphism}
of $\partial\mathcal{T}$. The Neretin group $\mathcal{N}_{d}$ is
defined as the group of all almost automorphism of $\mathcal{T}$.
Equivalently, $\mathcal{N}_{d}$ is the group of all spheromorphisms
of $\partial\mathcal{T}.$ For more detailed exposition see for example
\cite{GL-Neretin}. 

The group $\mathcal{N}_{d}$ can be viewed as the topological full
group of ${\rm Aut}(\mathcal{T})\curvearrowright\partial\mathcal{T}$.
Given a group $G$ acting on a topological space $X$, the \emph{topological
full group} of $G\curvearrowright X$ consists of all homeomorphisms
$\varphi$ of $X$ such that for any $x\in X$, there exists a neighborhood
$U$ of $x$ and an element $g\in G$ such that $\varphi|_{U}=g|_{U}$.
The topology on $\mathcal{N}_{d}$ is defined such that the inclusion
${\rm Aut}(\mathcal{T})\hookrightarrow\mathcal{N}_{d}$ is open and
continuous.

In \cite{caprace-de-dedts}, it is shown that $\mathcal{N}_{d}$ is
compactly generated: indeed it contains a dense copy of the Higman-Thompson
group $V_{d,d+1}$, which is finitely generated. We now describe the
embedded Higman-Thompson group following \cite{caprace-de-dedts}.
For general reference on Higman-Thompson groups, see for instance
\cite{brown}. Let $\mathsf{\mathcal{T}}_{d,k}$ be the rooted tree
where the root $v_{0}$ has $k$ children and all the other vertices
have $d$ children. For each vertex $v$, fix a local order $<_{v}$,
which is a total order on the children of $v$. Such a collection
of total orders $\left\{ <_{v}\right\} $ is referred to as a plane
order, as it specifies an embedding of the tree $\mathsf{\mathcal{T}}_{d.k}$
in $\mathbb{R}^{2}$, where the root $v_{0}$ is drawn at the origin,
and the children of a vertex are drawn below the parent, arranged
from left to right according to the order. An almost automorphism
is \emph{locally order preserving} if it can be represented by a triple
$(A,B,\varphi)$ where for each vertex $v\in\mathcal{T}_{d,k}\setminus A$,
the restriction of $\varphi$ on the children of $v$ preserves the
order. The subgroup of ${\rm AAut}(\mathcal{T}_{d,k})$ which consists
of locally order preserving elements is the \emph{Higman-Thompson
group} $V_{d,k}$. Returning to the $(d+1)$-regular tree $\mathcal{T},$
we have that a plane order on $\mathcal{T}$ gives an embedding of
the group $V_{d,d+1}$ as a dense subgroup of $\mathcal{N}_{d}$. 

\emph{Coloured Neretin groups} are introduced and investigated in
\cite{Lederle}. Take a closed subgroup $G<{\rm Aut}(\mathcal{T})$
and let $\mathsf{F}(G)$ be the topological full group of the action
$G\curvearrowright\partial\mathcal{T}$. When $G$ has Tits' independence
property, there exists a unique group topology on $\mathsf{F}(G)$
such that the inclusion $G\hookrightarrow\mathsf{F}(G)$ is open and
continuous, see \cite[Proposition 2.22]{Lederle}. Equipped with this
topology, $\mathsf{F}(G)$ is a t.d.l.c.\@ group containing $G$
as an open subgroup. 

Consider the case where $G$ is a universal group acting on $\mathcal{T}$
with a prescribed local action in the sense of Burger-Mozes \cite{Burger-Mozes}.
Recall that we have fixed a legal colouring of the tree $\mathcal{T}$.
Given a subgroup $F<{\rm Sym}(D)$, the \emph{Burger-Mozes' universal
group} $U(F)$ is defined as the subgroup of ${\rm Aut}(\mathcal{T})$
which consists of elements whose local action at every vertex is in
$F$. More precisely, at any vertex $v$ of $\mathcal{T}$, an automorphism
$g\in{\rm Aut}(\mathcal{T})$ induces a bijection $g_{v}:E(v)\to E(g(v))$,
where $E(v)$ denotes edges incident to $v$. The bijection $g_{v}$
gives rise to a local permutation of colours given by $\sigma(g,v)={\rm col}_{v}^{-1}\circ g_{v}\circ{\rm col}_{g(v)}$
in ${\rm Sym}(D)$. The group $U(F)$ consists of all automorphisms
$g\in{\rm Aut}(\mathcal{T})$ such that $\sigma(g,v)\in F$ for all
$v\in\mathcal{T}$. 

Denote by $\mathcal{N}_{F}$ the topological full group of the action
$U(F)\curvearrowright\partial\mathcal{T}$, equipped with the unique
group topology such that $U(F)\hookrightarrow\mathcal{N}_{F}$ is
open and continuous. We refer to $\mathcal{N}_{F}$ as the \emph{coloured
Neretin group associated with} $F$. Elements of $\mathcal{N}_{F}$
are called $U(F)$-almost automorphisms and each element $g\in\mathcal{N}_{F}$
can be represented by a triple $\left(A,B,\varphi\right)$, where
$A,B$ are complete finite subtrees with $|\partial A|=\left|\partial B\right|$
and $\varphi$ is a forest isomorphism $\mathcal{T}\setminus A\to\mathcal{T}\setminus B$
such that for each leaf $v\in\partial A$, there exists an element
$h_{v}\in U(F)$ such that $\varphi|_{\mathcal{T}_{v}}=h_{v}|_{\mathcal{T}_{v}}.$ 

Given $F$, denote by $\left\{ D^{(0)},\ldots,D^{(l)}\right\} $ the
$F$-orbits in $D=\{0,1,\ldots,d\}$. The group $F$ induces a labeling
on the vertices of the tree $\mathcal{T}$ except at the root $v_{0}$:
for any vertex $v\neq v_{0}$, define $\ell_{F}(v)$ as the $F$-orbit
of ${\rm col}(e)$, where $e$ is the edge connecting $v$ to its
parent. Suppose a plane order on $\mathcal{T}$ is given, we say an
almost automorphism in $V_{d,d+1}$ is $\ell_{F}$-\emph{label preserving
}if it can be represented by a triple $(A,B,\varphi)$ where $\varphi$
is a locally order preserving forest isomorphism and for each leaf
$v$ of $A$, $\ell_{F}(v)=\ell_{F}(\varphi(v))$.

By \cite[Proposition 3.14]{Lederle}, there exists a plane order on
the tree $\mathcal{T}$ which is compatible with the vertex labeling
$\ell_{F}$ such that the subgroup of almost automorphisms that are
locally order preserving and $\ell_{F}$-label preserving is a subgroup
of $\mathcal{N}_{F}$. Denote this group by $V_{F}$, $V_{F}=V_{d,d+1}\cap\mathcal{N}_{F}$.
By \cite[Theorem 4.1]{Lederle}, $V_{F}$ is dense in $\mathcal{N}_{F}$.
For a given $F$, we fix such a compatible plane order. The group
$V_{F}$ is analogous to the Higman-Thompson groups, its isomorphism
class only depends on the size of $F$-orbits in $D$. By \cite[Theorem 3.9]{Lederle}
$V_{F}$ can be identified as the topological full group of a one-sided
irreducible shift of finite type, which is introduced by Matui in
\cite{Matui2015}. 

\section{Induced IRSs and events in sub-quotients\label{sec:Induced}}

Let $A$ be a complete finite subtree of $\mathcal{T}.$ Recall that
$A$ is called complete if it contains the root $v_{0}$ and if a
vertex $v\in A$ is not a leaf, then all of its children are contained
in $A$. Denote by $B_{n}(A)$ the subtree with vertices within distance
$n$ to $A$, in other words it is the subtree which contains $A$
and trees of height $n$ rooted at the leaves of $A$. 

Denote by $O_{F}^{A}(n)$ the subgroup which consists of elements
in $\mathcal{N}_{F}$ which can be represented by a triple $\left(B_{n}(A),B_{n}(A),\varphi\right)$,
where $\varphi$ is a forest isomorphism $\mathcal{T}\setminus B_{n}(A)\to\mathcal{T}\setminus B_{n}(A)$
such that for each leaf $v\in\partial B_{n}(A)$, there exists an
element $h_{v}\in U(F)$ such that $\varphi|_{\mathcal{T}_{v}}=h_{v}|_{\mathcal{T}_{v}}.$
It is an open compact subgroup of $\mathcal{N}_{F}$. Let $O_{F}^{A}$
be the increasing union
\[
O_{F}^{A}:=\bigcup_{n=0}^{\infty}O_{F}^{A}(n).
\]
For instance, when $A=\{v_{0}\}$, the corresponding group $O_{F}^{\{v_{0}\}}$
is the open subgroup $O$ considered in \cite{BCGM,Lederle}. 

The group $O_{F}^{A}(n)$ permutes the leaves of the subtree $B_{n}(A)$
and the kernel of this action is the pointwise stabilizer of $B_{n}(A)$
in $U(F)$, which we denote by
\[
U_{F}^{A}(n):=\{g\in U(F):\ x\cdot g=x\mbox{ for all }x\in B_{n}(A)\}.
\]
Note that $U_{F}^{A}(n)$ is open and compact. Denote by $S_{F}^{A}(n)$
the quotient group $O_{F}^{A}(n)/U_{F}^{A}(n)$ and $\pi_{n}$ the
projection
\[
\pi_{n}:O_{F}^{A}(n)\to S_{F}^{A}(n).
\]
When $F$ is transitive on $D$, the quotient $S_{F}^{A}(n)$ is isomorphic
to the symmetric group ${\rm Sym}(\partial B_{n}(A))$. For general
$F$, recall that we denote by $\left\{ D^{(0)},\ldots,D^{(l)}\right\} $
the $F$-orbits on $D$ and $\ell_{F}$ be the labeling associated
to $F$-orbits on $\mathcal{T}\setminus\{v_{0}\}$. Then the quotient
group $S_{F}^{A}(n)$ is isomorphic to $\prod_{i=0}^{l}{\rm Sym}\left(D_{n,A}^{(i)}\right)$,
where $D_{n,A}^{(i)}=\{v\in\partial B_{n}(A):\ \ell_{F}(v)=D^{(i)}\}$. 

Given a finite complete subtree $A$ and two distinct leaves $u,v$
of $A$ with $\ell_{F}(u)=\ell_{F}(v)$, consider the event in ${\rm Sub}(\mathcal{N}_{F})$
\begin{equation}
\Theta_{u,v}^{A}:=\left\{ H\in{\rm Sub}\left(\mathcal{N}_{F}\right):\ \exists h\in H\cap O_{F}^{A}(0)\mbox{ s.t.}\ v=u\cdot\pi_{0}(h)\right\} .\label{eq:Auv}
\end{equation}
Since the set of $\left\{ g\in O_{F}^{A}(0):\ v=u\cdot\pi_{0}(g)\right\} $
is open in $\mathcal{N}_{F}$, we have that $\Theta_{u,v}^{A}$ is
an open subset in ${\rm Sub}\left(\mathcal{N}_{F}\right)$. This event
is similar to the event $\Omega_{U,V}$ considered in Section \ref{sec:Effective}. 

Let $\mu$ be an ergodic IRS of $\mathcal{N}_{F}$, $\mu\neq\delta_{\{id\}}$.
We will verify that for $\mu$-a.e. $H$, there exists a finite complete
tree $A$ and two distinct leaves $u,v\in\partial A$ such that $H\in\Theta_{u,v}^{A}$,
see Lemma \ref{cover}. Thus in what follows we focus on these complete
finite subtree $A$ and leaves $u\neq v$ in $\partial A$ such that
$\mu\left(\Theta_{u,v}^{A}\right)>0$. 

Given a leaf $u$ of a finite complete subtree $A$, denote by $C_{u}^{n}$
the set of vertices $x\in\mathcal{T}$ such that $u$ lies on the
geodesic from the root $v_{0}$ to $x$ and $d_{\mathcal{T}}(u,x)=n$.
In other words, $C_{u}^{n}$ consists of vertices of distance $n$
to $u$ in the subtree $\mathcal{T}_{u}$ rooted at $u$. 

Now suppose $A,u,v$ are such that $\mu\left(\Theta_{u,v}^{A}\right)>0$.
Expand the subtree $A$ to $B_{n}(A)$ and consider in the finite
group $S_{F}^{A}(n)$ the event 
\begin{equation}
\Theta_{u,v}^{A,n}:=\left\{ H\le S_{F}^{A}(n):\ \exists h\in H\cap\pi_{n}\left(O_{F}^{A}(0)\right)\mbox{ s.t. }C_{v}^{n}=C_{u}^{n}\cdot h\right\} .\label{eq:Auvn}
\end{equation}
It's clear by definitions of the events that for any $n\ge0$,
\[
\Theta_{u,v}^{A}\subseteq\left\{ H\in{\rm Sub}\left(\mathcal{N}_{F}\right):\ \pi_{n}\left(H\cap O_{F}^{A}(n)\right)\in\Theta_{u,v}^{A,n}\right\} .
\]
Note that the maps $H\mapsto H\cap O_{F}^{A}(n)$ and $H\mapsto\ \pi_{n}\left(H\cap O_{F}^{A}(n)\right)$
are continuous. Denote by $\bar{\mu}_{n}^{A}$ the induced IRS in
the finite group $S_{F}^{A}(n)$, that is, the pushforward of $\mu$
under $H\mapsto\pi_{n}\left(H\cap O_{F}^{A}(n)\right)$. Suppose $\mu\left(\Theta_{u,v}^{A}\right)>0$,
then we have for all $n$, 
\[
\bar{\mu}_{n}^{A}\left(\Theta_{u,v}^{A,n}\right)\ge\mu\left(\Theta_{u,v}^{A}\right)>0.
\]
This uniform lower bound, independent of $n$, on the probability
of containing a specific kind of almost automophisms, is the starting
point of our argument. We will show, by combining general lemmas in
Section \ref{sec:Effective} and properties of finite symmetric groups,
that this lower bound forces the finitary IRS $\bar{\mu}_{n}^{A}$
to charge groups that contain a \textquotedbl large\textquotedbl{}
alternating subgroup of $S_{F}^{A}(n)$. Proposition \ref{contain}
will be shown by applying the Borel-Cantelli lemma to combine the
estimates in each level $n$.

\section{Proof of Theorem \ref{main} given Proposition \ref{contain} \label{sec:Proof-of-Theorem}}

In this section we explain how to deduce Theorem \ref{main} from
Proposition \ref{contain}. Recall that in the previous section, we
have defined the countable collection of events $\left\{ \Theta_{u,v}^{A}\right\} $
as in (\ref{eq:Auv}), where $u,v$ goes over all pairs of distinct
vertices in $\partial A$ and $A$ goes over all finite complete subtrees
of $\mathcal{T}$. 

To proceed, we first show that if $\mu$ is an ergodic IRS of $\mathcal{N}_{F}$
such that $\mu\neq\delta_{\{id\}}$, then for $\mu$-a.e. $H$, there
exists some $A$ such that $H\cap O_{F}^{A}\neq\{id\}$. When $H$
is of finite co-volume, this is obvious because $O_{F}^{A}$ has infinite
Haar measure. For essentially the same reason, an ergodic IRS that
is concentrated at the trivial group cannot intersect all $O_{F}^{A}$
trivially:

\begin{lemma}\label{cover}

Let $\mu$ be an ergodic IRS of $\mathcal{N}_{F}$, $\mu\neq\delta_{\{id\}}$.
Then for $\mu$-a.e. $H$, there exists a finite complete subtree
$A$ such that $H\cap O_{F}^{A}\neq\{id\}$.

\end{lemma}

\begin{proof}

Suppose $H$ is such that $H\cap O_{F}^{A}=\{id\}$ for all $A$.
It follows that for any $A,B$ with $\left|\partial A\right|=\left|\partial B\right|$,
there is at most one element in $H$ that can be represented by a
triple of the form $\left(A,B,\varphi\right)$. Indeed, if there are
two distinct elements $h_{1},h_{2}$ of this form, then $h_{1}h_{2}^{-1}$
would be a non-identity element in $H\cap O_{F}^{A}$, contradicting
$H\cap O_{F}^{A}=\{id\}$. 

Note that the union $\bigcup_{A}\left\{ H:H\cap O_{F}^{A}\neq\{id\}\right\} $,
over all finite complete subtrees $A$, is invariant under conjugation
by $\mathcal{N}_{F}$. Thus by ergodicity the $\mu$-measure of this
union is either $0$ or $1$. We argue by contradiction and assume
from now on 
\begin{equation}
\mu\left(\bigcup_{A}\left\{ H:H\cap O_{F}^{A}\neq\{id\}\right\} \right)=0.\label{eq:non-intersect}
\end{equation}

Given two finite complete subtrees $A,B$ with $|\partial A|=|\partial B|$,
a point $u\in\partial A$ and $v\in\partial B$, denote by $W(A,B:u,v)$
the set of $U(F)$-almost automorphisms that can be written as a product
$\Psi(A,B,\sigma)$, where $\Psi\in O_{F}^{A}$ with $u\cdot\Psi=u$
for all $u\in\partial A$ and $\sigma$ is a locally order preserving
forest isomorphism $\mathcal{T}\setminus A\to\mathcal{T}\setminus B$
with $u\sigma=v$. Take the collection $\left\{ W(A,B:u,v)\right\} $
where the pairs $u,v$ are such that $u\in\partial A$, $v\in\partial B$
and $C_{u}\cap C_{v}=\emptyset$; and $A,B$ go over all finite complete
subtrees with $|\partial A|=|\partial B|$. The corresponding collection
$\left\{ H\in{\rm Sub}(\mathcal{N}_{F}):H\cap W(A,B:u,v)\neq\emptyset\right\} $
form an open cover of ${\rm Sub}\left(\mathcal{N}_{F}\right)\setminus\{\{id\}\}$.
Indeed, to verify this claim, let $g\in\mathcal{N}_{F}$ be any non-identity
element and $(A_{0},B_{0},\varphi_{0})$ be a representing triple
for $g$. Then there exists disjoint clopen subsets $V_{1},V_{2}\subseteq\partial\mathcal{T}$
such that $V_{2}=V_{1}\cdot g$. Expand the trees $A_{0}$ and $B_{0}$
to sufficiently large levels, we may represent $g$ by a triple $(A_{1},B_{1},\varphi_{1})$
such that there exists a vertex $u\in\partial A_{1}$ with $C_{u}\subseteq V_{1}$.
It follows that $g\in W(A_{1},B_{1}:u,\varphi_{1}(u))$ where $C_{u}\cap C_{\varphi_{1}(u)}=\emptyset$. 

Since the cover of ${\rm Sub}\left(\mathcal{N}_{F}\right)\setminus\{\{id\}\}$
in the previous paragraph is countable, there must exist some $A,B,u,v$
such that 
\[
\mu\left(\left\{ H:H\cap W(A,B:u,v)\neq\emptyset\right\} \right)>0.
\]

Recall the fact that if $H\cap W(A,B:u,v)\neq\emptyset$ and $H\cap O_{F}^{A}=\{id\}$,
then the intersection $H\cap W(A,B:u,v)$ consists of a unique element.
In this case, we write the unique element of $H\cap W(A,B:u,v)$ as
$\left(\Psi_{x}^{H}\right)_{x\in\partial A}\sigma^{H}$, where $\left(\Psi_{x}^{H}\right)_{x\in\partial A}\in O_{F}^{A}$
and $\sigma^{H}$ is a bijection from $\partial A$ to $\partial B$
with $\sigma^{H}(u)=v$. Recall that we have assumed (\ref{eq:non-intersect}).
Denote by $\eta_{u}$ the conditional distribution of the element
$\Psi_{u}^{H}$ given that $H\cap W(A,B:u,v)\neq\emptyset$. We now
repeat the same kind of argument as in the proof of Lemma \ref{E1}.
Take any element $g\in R_{O_{F}^{B}}(C_{v})$, then under conjugation
by $g$, we have 
\begin{equation}
\Psi_{u}^{g^{-1}Hg}=\Psi_{u}^{H}\sigma g\sigma^{-1}.\label{eq:trans}
\end{equation}
Note that the set $\left\{ H:H\cap W(A,B:u,v)\neq\emptyset\right\} $
is invariant under conjugation by $R_{O_{F}^{B}}(C_{v})$. It follows
from conjugacy invariance of $\mu$ and (\ref{eq:trans}) that the
distribution $\eta_{u}$ is invariant under right translation by $R_{O_{F}^{A}}(C_{u})$.
However this is impossible because $R_{O_{F}^{A}}(C_{u})$ does not
admit a finite right Haar measure. We have reached a contradiction
and therefore (\ref{eq:non-intersect}) is false. Instead, we have
\[
\mu\left(\bigcup_{A}\left\{ H:H\cap O_{F}^{A}\neq\{id\}\right\} \right)=1.
\]

\end{proof}

Now we assume Proposition \ref{contain} and show Theorem \ref{main}
by reducing the problem to the Higman-Thompson type group $V_{F}$
densely embedded in $\mathcal{N}_{F}$. 

Since $V_{F}$ is countable, the Chabauty space ${\rm Sub}(V_{F})$
is equipped with the topology inherited from the product topology
on $\{0,1\}^{V_{F}}.$ The intersection map ${\rm Sub}(\mathcal{N}_{F})\to{\rm Sub}(V_{F})$
given by $H\mapsto H\cap V_{F}$ is Borel measurable. Thus given an
IRS $\mu$ of $\mathcal{N}_{F}$, we can consider its pushforward
$\mu_{V_{F}}$ under $H\mapsto H\cap V_{F}$. 

A priori, an ergodic IRS of $\mathcal{N}_{F}$ may intersect with
$V_{F}$ trivially. However this cannot happen unless $H=\{id\}$
because of Lemma \ref{cover} and Proposition \ref{contain}. The
intersection $V_{F}\cap R_{O_{F}^{A}}(U)$ is nontrivial for any choice
of $A$ and open set $U$, indeed, $V_{F}\cap R_{O_{F}^{A}}(U)$ is
dense in $R_{O_{F}^{A}}(U)$, where $R_{O_{F}^{A}}(U)$ is equipped
with the subspace topology inherited from the natural inclusion $R_{O_{F}^{A}}\left(U\right)\hookrightarrow\mathcal{N}_{F}$.
Thus we complete the proof of Theorem \ref{main} by combining Proposition
\ref{contain} and the result that $V_{F}'$ has no non-trivial IRSs. 

\begin{proof}[Proof of Theorem \ref{main} given Proposition \ref{contain}]

Let $\mu$ be an ergodic IRS of $\mathcal{N}_{d}$, $\mu\neq\delta_{\{id\}}$.
By Lemma \ref{cover} and Proposition \ref{contain} (when $F$ is
transitive on $D$ one can use the special case Proposition \ref{containRO}),
we have that the collection of subgroups $H$ with the property that
there exists a complete finite tree $A$ and a non-empty open set
$U\subseteq\partial\mathcal{T}$ such that $H\ge R_{O_{F}^{A}}(U)'$
has $\mu$-measure $1$. Since $R_{O_{F}^{A}}(U)'$ contains non-trivial
locally order preserving $U(F)$-almost automorphisms, we have that
$V_{F}'\cap H\neq\{id\}$. 

Consider the induced IRS $\mu_{V_{F}}$ which is the pushforward of
$\mu$ under the intersection map $H\mapsto H\cap V_{F}$. Then by
\cite[Corollary 3.9]{Dudko-Medynets2}, the ergodic decomposition
of $\mu_{V_{F}}$ is of a convex combination of the form 
\[
\mu_{V_{F}}=\lambda_{0}\delta_{\{id\}}+\sum_{i=1}^{p}\lambda_{i}\delta_{L_{i}},
\]
where $p$ is finite and $L_{i}$ is a normal subgroup of $V_{F}$
that contains $V_{F}'$, see also \cite[Corollary 5.4]{rigidstab}
for an alternative proof. As explained in the previous paragraph,
Lemma \ref{cover} and Proposition \ref{contain} implies that $\mu$-a.e.
$H$ satisfies $V_{F}'\cap H\neq\{id\}$. Therefore $\lambda_{0}=0$
and $\mu$-a.e. $H$, $V_{F}\cap H\ge V_{F}'$. Since $H$ is closed,
it follows that $H$ contains the closure of $V_{F}'$. Since the
closure of $V_{F}'$ contains the commutator subgroup $\mathcal{N}_{F}'$,
we have proved the statement. 

\end{proof}

The proof of Proposition \ref{contain} occupies the next three sections.

\section{Tree automorphism orbits versus random permutations\label{sec:Tree-match}}

Consider a rooted tree $\mathsf{T}_{d,q}$, $d,q\ge2$, where the
root $o$ has $q$ children and the rest of the vertices have $d$
children. Denote by $W={\rm Aut}\left(\mathsf{T}_{d.q}\right)$ the
group of \emph{rooted} tree automorphisms of $\mathsf{T}_{d,q}$.
Note that $W$ has the structure of a semi-direct product 
\[
W=\left(\oplus_{v\in\mathsf{L}_{1}}{\rm Aut}(\mathsf{T}_{v})\right)\rtimes{\rm Sym}\left(\mathsf{L}_{1}\right),
\]
where $\mathsf{T}_{v}$ is the subtree rooted at $v$ and ${\rm Aut}(\mathsf{T}_{v})$
is the group of rooted tree automorphisms of $\mathsf{T}_{v}$. The
font $\mathsf{T}$ is used in this section to emphasize that the tree
is rooted and ${\rm Aut}(\mathsf{T}_{v})$ is the group of rooted
automorphisms. 

Write $\mathsf{L}_{n}$ for the level $n$ vertices with respect to
the root $o$, that is, $\mathsf{L}_{n}$ consists of vertices of
$\mathsf{T}_{d,q}$ such that $d(o,v)=n$. A vertex in $\mathsf{L}_{n}$
is encoded as a string $v=v_{1}\ldots v_{n}$, where $v_{1}\in\{0,\ldots,q-1\}$
and $v_{j}\in\left\{ 0,\ldots,d-1\right\} $ for $2\le j\le n$. For
two subset $E,F$ of $\mathsf{L}_{n}$, we write 
\[
E\sim_{W}F
\]
if $F$ is in the $W$-orbit of $E$, that is, $E\sim_{W}F$ if and
only if there exists $g\in W$ such that $F=E\cdot g$. 

In this auxiliary section we estimate the probability that two randomly
chosen subsets are in the same orbit of $W$. Such estimates will
be useful in the next sections to rule out certain cases of intransitivity
or imprimitivity. As in the previous section, denote by $C_{u}^{n}$
the vertices in the subtree rooted at $u$ of distance $n$ to $u$. 

\begin{lemma}\label{treematch}

In the rooted tree $\mathsf{T}_{d,q}$, let $u,v$ be two distinct
vertices in $\mathsf{L}_{1}$. Let $n,k$ be integers such that $n\ge2$
and $2\le k\le d^{n}/2$. Choose a set $E$ of size $k$ uniformly
random from $C_{u}^{n}$ and independently choose a set $F$ of size
$k$ uniformly random from $C_{v}^{n}$. Then for any $\delta>0$,
there exists constants $C,c>0$ only depending on $\delta,d$, such
that for all such $n,k$, 
\[
\mathbb{P}\left(E\sim_{W}F\right)\le C\exp\left(-ck^{\frac{d-1}{2d}-\delta}\right).
\]

\end{lemma}

The lemma is shown by recursion down the tree. We use the following
well-known basic probability estimates. For $p,q\in(0,1)$, denote
by $H(q||p)$ the relative entropy (also called the Kullback--Leibler
divergence) between the Bernoulli distribution with parameter $q$
and the Bernoulli distribution with parameter $p$, that is, 
\begin{equation}
H(q||p)=q\log\frac{q}{p}+(1-q)\log\frac{1-q}{1-p}.\label{eq:relative}
\end{equation}
The relative entropy $H(q||p)$ is always non-negative and is zero
if and only if $q=p$.

\begin{fact}\label{kset}

Let $X$ be a finite set, $\boldsymbol{\sigma}$ a uniformly random
permutation in ${\rm Sym}(X)$. Let $U$ and $K$ be two non-empty
subset of $X$ and write $p=|U|/|X|$, $k=|K|$. Then for any $x>0$, 

\begin{align}
\mathbb{P}\left(\left|\left(K\cdot\boldsymbol{\sigma}\right)\cap U\right|>(p+x)k\right) & \le e^{-H(p+x||p)k},\nonumber \\
\mathbb{P}\left(\left|\left(K\cdot\boldsymbol{\sigma}\right)\cap U\right|<(p-x)k\right) & \le e^{-H(p-x||p)k}.\label{eq:rel-ent}
\end{align}
Moreover, suppose $k\le|X|/2$, there exists an absolute constant
$C>0$ such that for $i\in[pk/2,\frac{3}{2}pk]$,
\begin{equation}
\mathbb{P}(|\left(K\cdot\boldsymbol{\sigma}\right)\cap U|=i)\le C\sqrt{\frac{k}{i(k-i)}}.\label{eq:sup}
\end{equation}

\end{fact}

We include a proof for Fact \ref{kset} for the convenience of the
reader. Recall Stirling's approximation:
\begin{equation}
1\le\frac{n!}{\sqrt{2\pi n}\left(\frac{n}{e}\right)^{n}}\le\frac{e}{\sqrt{2\pi}}\mbox{ for all }n\ge1.\label{eq:stirling}
\end{equation}
\begin{proof}[Proof of Fact \ref{kset}]

List the elements of $K$ as $x_{1},\ldots,x_{k}$ and write $Z_{j}=\sum_{i=1}^{k}{\bf 1}_{U}\left(x_{i}\cdot\boldsymbol{\sigma}\right)$.
Then $Z_{k}=|\left(K\cdot\boldsymbol{\sigma}\right)\cap U|$. The
moment generating function of $Z_{k}$ satisfies that for $\lambda>0$,
\begin{align*}
\mathbb{E}\left[e^{\lambda Z_{k}}\right] & =\mathbb{E}\left[e^{\lambda Z_{k-1}}\mathbb{E}\left[e^{\lambda(Z_{k}-Z_{k-1})}|Z_{k-1}\right]\right]\\
 & =\mathbb{E}\left[e^{\lambda Z_{k-1}}\left(e^{\lambda}\frac{\left(|U|-Z_{k-1}\right)_{+}}{|X|-Z_{k-1}}+\frac{|X|-Z_{k-1}-\left(|U|-Z_{k-1}\right)_{+}}{|X|-Z_{k-1}}\right)\right]\\
 & \le\mathbb{E}\left[e^{\lambda Z_{k-1}}\left(e^{\lambda}p+1-p\right)\right].
\end{align*}
Iterate this inequality we have $\mathbb{E}\left[e^{\lambda Z_{k}}\right]\le\left(e^{\lambda}p+1-p\right)^{k}$.
By the Chernoff bound, we have
\[
\mathbb{P}\left(Z_{k}\ge(p+x)k\right)\le e^{-\lambda(p+x)}\mathbb{E}\left[e^{\lambda Z_{k}}\right]\le e^{-\lambda(p+x)}\left(e^{\lambda}p+(1-p)\right)^{k}.
\]
Optimize the choice of $\lambda$ we obtain the first inequality.
Similarly, $\mathbb{P}\left(Z_{k}\le(p-x)k\right)=\mathbb{P}\left(\left|\left(K\cdot\boldsymbol{\sigma}\right)\cap U^{c}\right|\ge(1-p+x)k\right)$.
Apply the first inequality to the set $U^{c}$, we obtain the second
inequality. 

For the last bound, write $u=|U|$, $x=|X|$. We have that by Stirling's
approximation (\ref{eq:stirling}),
\begin{align*}
\mathbb{P}(|\left(K\cdot\boldsymbol{\sigma}\right)\cap U| & =i)=\left(\begin{array}{c}
u\\
i
\end{array}\right)\left(\begin{array}{c}
x-u\\
k-i
\end{array}\right)/\left(\begin{array}{c}
x\\
k
\end{array}\right)\\
 & \le C\left(\frac{1-k/x}{\left(1-i/u\right)\left(1-(k-i)/(x-u)\right)}\frac{k}{i(k-i)}\right)^{1/2}\\
 & \cdot\exp\left(-H\left(\frac{i}{k}||p\right)k-H\left(\frac{u-i}{x-k}||p\right)(x-k)\right).
\end{align*}
The statement follows.

\end{proof}

\begin{proof}[Proof of Lemma \ref{treematch}]

Consider recursively down the subtrees rooted at $u,v$. In order
to have an element in $W$ which maps $E\to F$, it is necessary that
there exists a permutation $\gamma\in{\rm Sym}\left(\left\{ 0,\ldots,d-1\right\} \right)={\rm Sym}(d)$
such that for each child $ui$ of $u$, $E\cap C_{ui}^{n-1}\sim_{W}F\cap C_{v\gamma(i)}^{n-1}$.
Note that since the sets are chosen uniformly at random independently,
conditioned on the event $\left|E\cap C_{ui}^{n-1}\right|=\left|F\cap C_{v\gamma(i)}^{n-1}\right|=r$,
the distribution of $E\cap C_{ui}^{n-1}$ and $F\cap C_{v\gamma(i)}^{n-1}$
are independent, where $E\cap C_{ui}^{n-1}$ is distributed uniformly
on subsets of size $r$ of $C_{ui}^{n-1}$ and $F\cap C_{v\gamma(i)}^{n-1}$
is distributed uniformly on subsets of size $r$ of $C_{v\gamma(i)}^{n-1}$. 

The number $n$ is fixed throughout the calculation. For two distinct
vertices $y,z$ in level $\ell$, choose a set $A_{1}$ of size $r$
uniformly from $C_{y}^{n-\ell}$, and independently a set $A_{2}$
of size $r$ uniformly from $C_{z}^{n-\ell}$. It is clear that the
probability that $A_{1}\sim_{W}A_{2}$ depends only on the level $\ell$
and the size $r$, and we denote it by 
\begin{equation}
a(\ell,r)=\mathbb{P}(A_{1}\sim_{W}A_{2}).\label{eq:alr-2}
\end{equation}
Take a small constant $\epsilon<1/d^{2}$. We have
\begin{align*}
\mathbb{P}(A_{1}\sim_{W}A_{2}) & \le\mathbb{P}\left(\exists i\in\{0,\ldots,d-1\}\ \left|\left|A_{1}\cap C_{yi}^{n-\ell-1}\right|-\frac{r}{d}\right|>\epsilon r\right)\\
 & +\mathbb{P}\left(A_{1}\sim_{W}A_{2}\mbox{ and }\left|\left|A_{1}\cap C_{yi}^{n-\ell-1}\right|-\frac{r}{d}\right|\le\epsilon r\mbox{ for all }i\right)\\
 & :={\rm I}+{\rm II}.
\end{align*}
Write 
\[
\mathfrak{h}(\epsilon)=\max\left\{ H\left(\frac{1}{d}-\epsilon||\frac{1}{d}\right),H\left(\frac{1}{d}+\epsilon||\frac{1}{d}\right)\right\} .
\]
Then by (\ref{eq:rel-ent}), we have 
\[
{\rm I}\le2d\exp\left(-\mathfrak{h}(\epsilon)r\right).
\]
Then one step recursion to the children of $y$ and $z$ as in the
previous paragraph shows that 
\begin{equation}
{\rm II}\le\sum_{\gamma\in{\rm Sym}(d)}\sum_{\substack{\left|r_{i}-r/d\right|\le\epsilon r,\sum r_{i}=r}
}\mathbb{P}\left(\bigcap_{i=0}^{d-1}\left\{ \left|A_{1}\cap C_{yi}^{n-\ell-1}\right|=\left|A_{2}\cap C_{z\gamma(i)}^{n-\ell-1}\right|=r_{i}\right\} \right)\prod_{i=0}^{d-1}a(\ell+1,r_{i}).\label{eq:II}
\end{equation}
By independence we have 
\begin{align*}
 & \mathbb{P}\left(\bigcap_{i=0}^{d-1}\left\{ \left|A_{1}\cap C_{yi}^{n-\ell-1}\right|=\left|A_{2}\cap C_{z\gamma(i)}^{n-\ell-1}\right|=r_{i}\right\} \right)\\
= & \mathbb{P}\left(\bigcap_{i=0}^{d-1}\left\{ \left|A_{1}\cap C_{yi}^{n-\ell-1}\right|=r_{i}\right\} \right)\mathbb{P}\left(\bigcap_{i=0}^{d-1}\left\{ \left|A_{2}\cap C_{zi}^{n-\ell-1}\right|=r_{i}\right\} \right)\\
\le & \mathbb{P}\left(\bigcap_{i=0}^{d-1}\left\{ \left|A_{1}\cap C_{yi}^{n-\ell-1}\right|=r_{i}\right\} \right)C_{d,\epsilon}r^{-\frac{d-1}{2}},
\end{align*}
where in the second line we applied the bound (\ref{eq:sup}) $d-1$
times and the constant $C_{d,\epsilon}$ is $C^{d-1}\left(\frac{1}{d}-(d-1)\epsilon\right)^{-d+1}.$
Note that if $r>d^{n-\ell}/2$, then we should replace $r$ by $d^{n-\ell}-r$
in $a(\ell,r)$. For $r\le d^{n-\ell}/2$, set
\[
\tilde{a}(\ell,r)=\sup_{r\le s\le d^{n-\ell}/2}a(\ell,s).
\]
Plugging back in (\ref{eq:II}), we have 
\[
{\rm II}\le d!C_{d,\epsilon}r^{-\frac{d-1}{2}}\tilde{a}\left(\ell+1,\frac{r}{d}-\epsilon r\right)^{d},
\]
Combine part I and II, we have
\[
\tilde{a}\left(\ell,r\right)\le d\exp\left(-\mathfrak{h}(\epsilon)r\right)+d!C_{d,\epsilon}r^{-\frac{d-1}{2}}\tilde{a}\left(\ell+1,\frac{r}{d}-\epsilon r\right)^{d}.
\]
Using the bound $(x+y)^{n}\le2^{n-1}(x^{n}+y^{n})$ we can iterate
this inequality. Start with $r$ where $d^{n-\ell}>2r$ and iterate
for $s$ steps, where $s$ is such that 
\begin{equation}
d!C_{d,\epsilon}^{d}r^{-\frac{d-1}{2}}\le\left(\frac{1}{4}\left(\frac{1}{d}-\epsilon\right)^{s}\right)^{d},\label{eq:s}
\end{equation}
then summing up the terms we have 
\[
a\left(\ell,r\right)\le C_{1}\exp\left(-\mathfrak{h}(\epsilon)r\left(1-d\epsilon\right)^{s-1}\right)+2^{-d^{s-1}},
\]
where $C_{1}$ is a constant depending only on $d,\epsilon$. Given
a $\delta>0$, choose $\epsilon$ sufficiently small and $s$ the
largest integer satisfying (\ref{eq:s}), we conclude that for $r\le d^{n-\ell}/2$,
\begin{equation}
a(\ell,r)\le C\exp\left(-cr^{\frac{d-1}{2d}-\delta}\right).\label{eq:alr}
\end{equation}
The statement is given by taking $\ell=1$ in (\ref{eq:alr}). 

\end{proof}

We deduce two corollaries from Fact \ref{kset} and Lemma \ref{treematch},
which will be used in the next section. 

\begin{corollary}\label{cut1}

In the rooted tree $\mathsf{T}_{d,q}$, let $u,v$ be two distinct
vertices in $\mathsf{L}_{1}$. Let $\boldsymbol{\sigma}$ be a uniform
random permutation in ${\rm Sym}(\mathsf{L}_{n})$. Let $K$ be a
subset of $\mathsf{L}_{n}$ with $k=|K|\le|\mathsf{L}_{n}|/2$. Then
for any $\delta>0$, we have
\[
\mathbb{P}\left(\left(K\cdot\boldsymbol{\sigma}\right)\cap C_{u}^{n}\sim_{W}\left(K\cdot\boldsymbol{\sigma}\right)\cap C_{v}^{n}\right)\le C\exp\left(-ck^{\frac{d-1}{2d}-\delta}\right),
\]
where $C,c>0$ are constants that only depends on $\delta,d$ and
$q$. 

\end{corollary}

\begin{proof}

In order to have an element in $W$ which maps $\left(K\cdot\boldsymbol{\sigma}\right)\cap C_{u}^{n}\to\left(K\cdot\boldsymbol{\sigma}\right)\cap C_{v}^{n}$,
it is necessary that they are of equal sizes. Since $\sigma$ is uniform,
conditioned on the event $\left|\left(K\cdot\sigma\right)\cap C_{u}^{n}\right|=\left|\left(K\cdot\sigma\right)\cap C_{v}^{n}\right|=r$,
the distribution of the sets $\left(K\cdot\sigma\right)\cap C_{u}^{n}$
and $\left(K\cdot\sigma\right)\cap C_{v}^{n}$ are independent, where
each $\left(K\cdot\sigma\right)\cap C_{x}^{n}$ is distributed uniformly
on subsets of size $r$ of $C_{x}^{n}$, $x\in\{u,v\}$. Therefore
we have 
\begin{multline*}
\mathbb{P}\left(\left(K\cdot\boldsymbol{\sigma}\right)\cap C_{u}^{n}\sim_{W}\left(K\cdot\boldsymbol{\sigma}\right)\cap C_{v}^{n}\right)\\
=\sum_{r\le k/2}\mathbb{P}\left(\left|\left(K\cdot\boldsymbol{\sigma}\right)\cap C_{u}^{n}\right|=\left|\left(K\cdot\boldsymbol{\sigma}\right)\cap C_{v}^{n}\right|=r\right)p(r,n),
\end{multline*}
where $p(r,n)=\mathbb{P}\left(E\sim_{W}F\right)$, the set $E$ is
a uniformly random subset of size $r$ in $C_{u}^{n}$ and $F$ is
an independent uniformly random subset of size $r$ in $C_{v}^{n}$.

By Fact \ref{kset}, the size of $\left|\left(K\cdot\boldsymbol{\sigma}\right)\cap C_{u}^{n}\right|$
is concentrated around $k/q$. Thus, apply Fact \ref{kset} and Lemma
\ref{treematch}, we have for any $\epsilon>0$,
\begin{multline*}
\mathbb{P}\left(\left(K\cdot\boldsymbol{\sigma}\right)\cap C_{u}^{n}\sim_{W}\left(K\cdot\boldsymbol{\sigma}\right)\cap C_{v}^{n}\right)\\
\le\exp\left(-H(1/q+\epsilon||1/q)k\right)+\exp\left(-H(1/q-\epsilon||1/q)k\right)+C\exp\left(-c\left(\frac{k}{q}-\epsilon k\right)^{\frac{d-1}{2d}-\delta}\right).
\end{multline*}
Choosing for example $\epsilon=\frac{1}{2q}$, we obtain the statement. 

\end{proof}

\begin{corollary}\label{cut2}

In the rooted tree $\mathsf{T}_{d,q}$, let $u,v$ be two distinct
vertices in $\mathsf{L}_{1}$. Let $\boldsymbol{\sigma}$ be a uniform
random permutation in ${\rm Sym}(\mathsf{L}_{n})$. Let $K_{1},K_{2}$
be two disjoint subsets of $\mathsf{L}_{n}$ with $\left|K_{1}\right|=\left|K_{2}\right|=k$.
Then
\[
\mathbb{P}\left(\left(K_{1}\cdot\boldsymbol{\sigma}\right)\cap C_{u}^{n}\sim_{W}\left(K_{2}\cdot\boldsymbol{\sigma}\right)\cap C_{v}^{n}\right)\le\exp\left(-c_{\delta}k^{\frac{d-1}{2d}-\delta}\right),
\]
where $c_{\delta}>0$ is a constant that only depends on $\delta,d$
and $q$. 

\end{corollary}

\begin{proof}

The proof is similar to Corollary \ref{cut1}. Since $\sigma$ is
uniform, conditioned on the event $\left|(K_{1}\cdot\sigma)\cap C_{u}^{n}\right|=\left|\left(K_{2}\cdot\sigma\right)\cap C_{v}^{n}\right|=r$,
the distributions of $\left(K_{1}\cdot\boldsymbol{\sigma}\right)\cap C_{u}^{n}$
and $\left(K_{2}\cdot\boldsymbol{\sigma}\right)\cap C_{v}^{n}$ are
independent and uniform in sets of size $r$ in $C_{u}^{n}$ and $C_{v}^{n}$
respectively. Let $p(r,n)$ be as in the proof of Corollary \ref{cut1}.
Then we have
\begin{align*}
 & \mathbb{P}\left(\left(K_{1}\cdot\boldsymbol{\sigma}\right)\cap C_{u}^{n}\sim_{W}\left(K_{2}\cdot\boldsymbol{\sigma}\right)\cap C_{v}^{n}\right)\\
= & \sum_{r}\mathbb{P}\left(\left|\left(K_{1}\cdot\boldsymbol{\sigma}\right)\cap C_{u}^{n}\right|=\left|\left(K_{2}\cdot\boldsymbol{\sigma}\right)\cap C_{v}^{n}\right|=r\right)p(r,n)\\
\le & \mathbb{P}\left(\left|\left(K_{1}\cdot\boldsymbol{\sigma}\right)\cap C_{u}^{n}\right|\notin\left[k/q-\epsilon k,k/q+\epsilon k\right]\right)+\max_{r\in\left[k/q-\epsilon k,k/q+\epsilon k\right]}p(r,n).
\end{align*}
The statement follows from Fact \ref{kset} and Lemma \ref{treematch}. 

\end{proof}

\section{Containment of rigid stabilizers when $F$ is transitive on $D$\label{sec:transitive}}

The goal of this section is to prove Proposition \ref{containRO},
assuming $F$ is transitive. The case of intransitive $F$ brings
in the complication that the quotient $S_{F}^{A}(n)$ is a product
of symmetric groups instead of ${\rm Sym}\left(\partial B_{n}(A)\right)$.
This is not hard to handle (see the next section), but for clarity
we present the argument for the transitive case first. 

Throughout this section $F$ is assumed to be transitive on $D$.
Let $\mu$ be an IRS of $\mathcal{N}_{F}$. Recall the setting and
notations in Section \ref{sec:Induced}. Suppose $A,u,v$ are such
that $\mu\left(\Theta_{u,v}^{A}\right)>0$, where $A$ is a finite
complete subtree, $u,v$ are two distinct vertices in $\partial A$,
and the event $\Theta_{u,v}^{A}$ is defined in (\ref{eq:Auv}). Fix
such a triple $A,u,v$. Go down $n$ more levels and consider the
induced IRS $\bar{\mu}_{n}^{A}$ in the finite group $S_{F}^{A}(n)={\rm Sym}\left(\partial B_{n}(A)\right)$,
that is, $\bar{\mu}_{n}^{A}$ is the pushforward of $\mu$ under the
map $H\mapsto\pi_{n}\left(H\cap O_{F}^{A}(n)\right)$. For $\Gamma<{\rm Sym}(\partial B_{n}(A))$,
let $\nu_{\Gamma}$ be the IRS of ${\rm Sym}(\partial B_{n}(A))$
which is uniform on conjugates of $\Gamma$. Denote the ergodic decomposition
of $\bar{\mu}_{n}^{A}$ by 
\[
\bar{\mu}_{n}^{A}=\sum_{i=1}^{I_{n}}\lambda_{i}\nu_{\Gamma_{i}},
\]
where $\nu_{\Gamma_{i}}$ is the IRS associated with the subgroup
$\Gamma_{i}<{\rm Sym}\left(\partial B_{n}(A)\right)$ and $I_{n}$
is a finite indexing set. 

Recall the event 
\[
\Theta_{u,v}^{A,n}=\left\{ H\le{\rm Sym}(\partial B_{n}(A)):\ \exists h\in H\cap\pi_{n}\left(O_{F}^{A}(0)\right)\mbox{ s.t. }C_{v}^{n}=C_{u}^{n}\cdot h\right\} 
\]
 as defined in (\ref{eq:Auvn}) and the fact that 
\[
\bar{\mu}_{n}^{A}\left(\Theta_{u,v}^{A,n}\right)\ge\mu\left(\Theta_{u,v}^{A}\right).
\]
Given a subgroup $\Gamma<{\rm Sym}(\partial B_{n}(A))$, consider
the probability $\nu_{\Gamma}\left(\Theta_{u,v}^{A,n}\right)$. We
want to show that if $\Gamma$ does not contain a large alternating
subgroup, then $\nu_{\Gamma}\left(\Theta_{u,v}^{A,n}\right)$ is small. 

One ingredient that goes into the bounds is the following direct consequence
of the subgroup index Lemma \ref{E1}, which is useful to subgroups
of relatively small index.

\begin{lemma}\label{index}

Suppose $Q$ is a subset of $\left\{ \gamma\in{\rm Sym}(X):\ U\cdot\gamma=V\right\} $.
Let $\Gamma<{\rm Sym}(X)$ be any subgroup. Then 
\[
\mathbb{P}_{\nu_{\Gamma}}(H\cap Q\neq\emptyset)\le\frac{|\Gamma|\cdot|Q{}_{U}|}{\left|{\rm Sym}(U)\right|},
\]
where $Q_{U}=\{\gamma|_{U}:\ \gamma\in Q\}.$

\end{lemma}

\begin{proof}

The rigid stabilizer of $U$ in ${\rm Sym}(X)$ is ${\rm Sym}(U)$.
Apply Lemma \ref{E1} to the IRS $\nu_{\Gamma}$, we have 
\[
\mathbb{P}_{\nu_{\Gamma}}(H\cap Q\neq\emptyset)\le\mathbb{E}_{\nu_{\Gamma}}\left[\frac{|Q_{U}|}{\left|{\rm Sym}(U):\left(\bar{H}_{U\to U}\cap{\rm Sym}\left(U\right)\right)\right|}\right].
\]
Since $\left|\bar{H}_{U\to U}\right|\le|H|=|\Gamma|$, we have 
\[
\left|{\rm Sym}(U):\left(\bar{H}_{U\to U}\cap{\rm Sym}\left(U\right)\right)\right|\ge\left|{\rm Sym}(U)\right|/|\Gamma|.
\]
The statement follows. 

\end{proof}

Write 
\[
q=|\partial A|\mbox{ and }k_{n}=qd^{n}=\left|\partial B_{n}(A)\right|.
\]
The group $O_{F}^{A}(0)$ is a subgroup of the semi-direct product
$W=\left(\oplus_{v\in\partial A}{\rm Aut}(\mathsf{T}_{v})\right)\rtimes{\rm Sym}\left(\partial A\right)$
as in the setting of Section \ref{sec:Tree-match}. We suppress reference
to $A$ in the notations $q,k_{n}$ and $W$, understanding that $A$
is fixed through the calculations. Denote by $Q_{u,v}^{A,n}$ the
subset 
\[
Q_{u,v}^{A,n}=\left\{ g\in\pi_{n}\left(O_{F}^{A}(0)\right):\ C_{v}^{n}=C_{u}^{n}\cdot g\right\} .
\]
 The size of the set of partial homeomorphisms $\left\{ g|_{C_{u}^{n}}:\ g\in Q_{u,v}^{A,n}\right\} $
is bounded by $|F|^{d^{n}}.$ Then by Lemma \ref{index}, we have
\begin{equation}
\nu_{\Gamma}\left(\Theta_{u,v}^{A,n}\right)=\mathbb{P}_{\nu_{\Gamma}}\left(H\cap Q_{u,v}^{A,n}\neq\emptyset\right)\le\frac{|\Gamma|\cdot|F|^{d^{n}}}{(d^{n})!}.\label{eq:size1}
\end{equation}
This shows that $\nu_{\Gamma}\left(Q_{u,v}^{A,n}\right)$ is small,
if the size of $\Gamma$ is much smaller than $(d^{n})!/|F|^{d^{n}}$.
Recall that the size of ${\rm Sym}(\partial B_{n}(A))$ is $\left(qd^{n}\right)$!.
As remarked earlier in the Introduction, this kind of bound is similar
to, but weaker than, the co-volume estimate used in the proof of absence
of lattices in \cite{BCGM}. 

Now consider in more detail the structure of $\Gamma$. The bounds
for $\nu_{\Gamma}\left(\Theta_{u,v}^{A,n}\right)$ are divided into
three cases below. The estimates we show here are far from being sharp,
but sufficient for the purpose of proving Proposition \ref{containRO}. 

To apply bounds in Section \ref{sec:Tree-match}, we fix a number
in $\left(0,\frac{d-1}{2d}\right)$, for instance, let
\[
\alpha=\frac{d-1}{4d}.
\]
In what follows, $\sigma$ denotes a random permutation with uniform
distribution in ${\rm Sym}\left(\partial B_{n}(A)\right)$. Denote
by $\mathsf{t}_{1},\ldots,\mathsf{t}_{r}$ the sizes of transitive
components of $\Gamma$ on $\partial B_{n}(A)$ and denote by $\mathsf{t}_{\Gamma}$
the maximum of $\mathsf{t}_{1},\ldots,\mathsf{t}_{r}$. 

\begin{lemma}[Case I: intransitive without giant component]\label{caseI}

Suppose $\Gamma$ is not transitive and denote by with $\mathsf{t}_{\Gamma}$
the maximum of sizes of transitive components. Then there exists a
constant $c,C>0$ depending only on $d,q$ such that 
\[
\nu_{\Gamma}\left(\Theta_{u,v}^{A,n}\right)\le C\exp\left(-c(k_{n}-\mathsf{t}_{\Gamma})^{\alpha}\right)+C\exp\left(-ck_{n}^{\alpha/2q}\right).
\]
\end{lemma}

\begin{proof}

Denote by $Y_{1},\ldots,Y_{r}$ the transitive components of $\Gamma$,
with $|Y_{1}|=\mathsf{t}_{\Gamma}$. The size of $\Gamma$ is at most
$\mathsf{t}_{1}!\ldots\mathsf{t}_{r}!$. Then by (\ref{eq:size1})
and Stirling's approximation (\ref{eq:stirling}), we have that if
$\mathsf{t}_{\Gamma}\le k_{n+m}^{1/2q}$, then there exists a constant
$C>0$ depending only on $q,d$ such that 
\[
\nu_{\Gamma}\left(\Theta_{u,v}^{A,n}\right)\le e^{-\frac{1}{2q}k_{n}\log\frac{k_{n}}{C}}.
\]

If $k_{n}^{1/2q}<\mathsf{t}_{\Gamma}\le k_{n}/2$, then in order for
$\sigma^{-1}\Gamma\sigma$ to contain an element $h\in\pi_{n}(O_{F}^{A}(0))$
with $C_{v}^{n}=C_{u}^{n}\cdot h$, it is necessary that for each
transitive component $Y_{i}\cdot\sigma$ of $\sigma^{-1}\Gamma\sigma$,
the intersections with $C_{u}$ and $C_{v}$ are in the same $O_{F}^{A}(0)$-orbit,
that is, for each $1\le i\le r$,
\[
\left(Y_{i}\cdot\sigma\right)\cap C_{u}^{n}\sim_{O_{F}^{A}(0)}\left(Y_{i}\cdot\sigma\right)\cap C_{v}^{n}.
\]
Since $O_{F}^{A}(0)\le W$, it is necessary then they are in the same
$W$-orbit. Thus by Corollary \ref{cut1} applied to the maximum component
$Y_{1}\cdot\sigma$, we have 
\begin{align*}
\nu_{\Gamma}\left(\Theta_{u,v}^{A,n}\right) & \le\mathbb{P}\left(\left(Y_{1}\cdot\sigma\right)\cap C_{u}^{n}\sim_{W}\left(Y_{1}\cdot\sigma\right)\cap C_{v}^{n}\right)\\
 & \le C\exp\left(-c\mathsf{t}_{\Gamma}^{\alpha}\right)\le C\exp\left(-ck_{n}^{\alpha/2q}\right).
\end{align*}

If $\mathsf{t}_{\Gamma}>k_{n}/2$, then consider the compliment of
$Y_{1}$ and the same reasoning as above implies that the intersection
of $\partial B_{n}(A)\setminus Y_{1}\cdot\sigma$ with $C_{u}^{n}$
and $C_{v}^{n}$ must be in the same $W$-orbit, therefore by Corollary
\ref{cut1}, we have 
\begin{align*}
\nu_{\Gamma}\left(\Theta_{u,v}^{A,n}\right) & \le\mathbb{P}\left(\left|\left(\partial B_{n}(A)\setminus Y_{1}\cdot\sigma\right)\cap C_{u}\right|\sim_{O_{n}}\left|\left(\partial B_{n}(A)\setminus Y_{1}\cdot\sigma\right)\cap C_{v}\right|\right)\\
 & \le C\exp\left(-c(k_{n}-\mathsf{t}_{\Gamma})^{\alpha}\right).
\end{align*}
The statement is obtained by combining these three cases. 

\end{proof}

Now consider the case where $\mathsf{t}_{\Gamma}$ is large and in
particular $\mathsf{t}_{\Gamma}>k_{n}/2$. We refer to the largest
transitive component as the giant component and denote it by $Y_{1}$.
Denote by $\bar{\Gamma}$ the projection of $\Gamma$ to permutations
of the giant component. If $\bar{\Gamma}$ is primitive but does not
contain ${\rm Alt}(Y_{1})$, then the size of $\bar{\Gamma}$ is small
and we can apply Lemma \ref{index} again. For our purposes it suffices
to use Praeger-Saxl's bound \cite{prager-saxl}: if $L\le{\rm Sym}(X)$
is primitive but does not contain ${\rm Alt}(X),$then $|L|\le4^{|X|}$.
Stronger bounds which are sub-exponential in $|X|$ are due to Babai
\cite{babai1,babai2}. Note that these results do not rely on classification
of finite simple groups. 

\begin{lemma}[Case II: primitive in the giant component but doesn't contain Alt]\label{caseII}

Suppose $\mathsf{t}_{\Gamma}>\left(1-\frac{1}{2q}\right)k_{n}$ and
the projection $\bar{\Gamma}$ in the giant component $Y_{1}$ is
primitive but doesn't contain ${\rm Alt}(Y_{1})$. Then there exists
a constant $C$ only depending on $q$ such that 
\[
\nu_{\Gamma}\left(\Theta_{u,v}^{A,n}\right)\le e^{-\frac{1}{2q}k_{n}\log\frac{k_{n}}{C}}.
\]

\end{lemma}

\begin{proof}

Under the assumptions of the lemma, by the Praeger-Saxl's bound, we
have 
\[
\left|\Gamma\right|\le\left(k_{n}-\mathsf{t}_{\Gamma}\right)!\cdot4^{\mathsf{t}_{\Gamma}}.
\]
The statement follows then from (\ref{eq:size1}). 

\end{proof}

It remains to consider the case of imprimitive $\bar{\Gamma}$. Given
$\bar{\Gamma}$, which is transitive on the giant component $Y_{1}$,
let $Z_{1},\ldots,Z_{\mathsf{p}}$ be the sets in the system of imprimitivity
for $\bar{\Gamma}$.

\begin{lemma}[Case III: imprimitive in the giant component]\label{caseIII}

Suppose $\mathsf{t}_{\Gamma}>(1-\frac{1}{2q})k_{n}$ and $\bar{\Gamma}$
is imprimitive in the giant transitive component $Y_{1}$. Then 
\[
\nu_{L}\left(\Theta_{u,v}^{A,n}\right)\le Ck_{n}\exp\left(-ck_{n}^{\alpha/3q}\right),
\]
where $c,C$ are constants only depending on $q$ and $d$. 

\end{lemma}

\begin{proof}

Denote by $\mathsf{p}_{\Gamma}$ the number of sets in the system
of imprimitivity. Write $b=\mathsf{t}_{\Gamma}/\mathsf{p}_{\Gamma}$,
that is, the cardinality of the block (domain of imprimitivity) $Z_{i}$.
The size of $\bar{\Gamma}$ is at most $\left(b!\right)^{\mathsf{p}_{\Gamma}}\mathsf{p}_{\Gamma}!$.
Thus by (\ref{eq:size1}) and Stirling's approximation, we have that
if $3q\le b\le k_{n}^{1/3q}$, then there exists a constant $C>0$
only depending on $q$ such that 
\[
\nu_{L}\left(\Theta_{u,v}^{A,n}\right)\le e^{-\frac{1}{6q}k_{n}\log\frac{k_{n}}{C}}.
\]

Next consider the case $k_{n}^{1/3q}<b\le\mathsf{t}_{\Gamma}/2$.
For $\sigma^{-1}\Gamma\sigma$ to contain an element $h\in\pi_{n}(O_{F}^{A}(0))$
that sends $C_{u}^{n}$ to $C_{v}^{n}$, it is necessary that for
each block $Z_{i}\cdot\sigma$, either $Z_{i}\cdot\sigma\cap C_{u}=\emptyset$,
or there exists a block $Z_{j}\cdot\sigma$ such that $Z_{i}\cdot\sigma\cap C_{u}^{n}\sim_{W}Z_{j}\cdot\sigma\cap C_{v}^{n}$.
It is allowed that $i=j$. The case of empty intersection can be viewed
as a special instance of $Z_{i}\cdot\sigma\cap C_{u}^{n}\sim_{W}Z_{i}\cdot\sigma\cap C_{v}^{n}$.
For $i\neq j$, apply Corollary \ref{cut2} to $Z_{i},Z_{j}$; and
for $i=j$, apply Corollary \ref{cut1} to $Z_{i}$. Then for $Z_{1}\cdot\sigma$,
take a union bound over $j$, we have that in this case
\begin{align*}
\nu_{L}\left(\Theta_{u,v}^{A,n}\right) & \le\sum_{1\le j\le\mathsf{p}_{\Gamma}}\mathbb{P}\left(Z_{1}\cdot\sigma\cap C_{u}^{n}\sim_{W}Z_{j}\cdot\sigma\cap C_{v}^{n}\right)\\
 & \le\mathsf{p}_{\Gamma}\cdot C\exp\left(-cb^{\alpha}\right)\le C\mathsf{p}_{\Gamma}\exp\left(-ck_{n}^{\alpha/3q}\right).
\end{align*}

Finally consider the case $b\le3q$, that is, the blocks are of bounded
size. Consider the blocks which are completely contained in $C_{u}$,
and the blocks which are completely contained in $C_{v}$. Denote
by $M_{x}(\sigma)$ the union of $Z_{i}\cdot\sigma$ such that $Z_{i}\cdot\sigma\subseteq C_{x}^{n}$,
$x\in\{u,v\}$. For $\sigma^{-1}\Gamma\sigma$ to contain an element
$h\in\pi_{n}(O_{F}^{A}(0))$ that sends $C_{u}^{n}$ to $C_{v}^{n}$,
it is necessary that $M_{u}(\sigma)\sim_{W}M_{v}(\sigma)$. Since
$\sigma$ is uniform, we have that conditioned on $\left|M_{u}(\sigma)\right|=\left|M_{v}(\sigma)\right|=r$,
the distribution of $M_{u}(\sigma)$ and $M_{v}(\sigma)$ are independent
and we are in the situation of Lemma \ref{treematch}. The probability
that a block $Z_{i}\cdot\sigma\subseteq C_{u}^{n}$ is bounded from
below by $\left(\frac{1}{q}\frac{k_{n}-q^{2}}{k_{n}-q}\right)^{b}$.
The Chernoff bound as in Fact \ref{kset} implies 
\[
\mathbb{P}\left(|M_{u}(\sigma)|\le\mathbb{E}|M_{u}(\sigma)|-\epsilon\mathsf{p}_{\Gamma}\right)\le e^{-c_{\epsilon}p_{\Gamma}},
\]
where $c$ is a constant depending on $q,\epsilon$. It follows from
Lemma \ref{sec:Tree-match} 
\[
\mathbb{P}\left(M_{u}(\sigma)\sim_{W}M_{v}(\sigma)\right)\le Ce^{-cp_{\Gamma}}+Ce^{-cp_{\Gamma}^{\alpha}},
\]
where the constants $c,C>0$ only depend on $q$ and $d$. The statement
follows from combining the three cases. 

\end{proof}

Next we combine these three cases. The right hand side of the bounds
in Case II and Case III are summable in $n$, while in the first case
the bound depends on the size $k_{n}-\mathsf{t}_{\Gamma}$. Choose
a sequence of increasing numbers $\left(\Delta_{n}\right)$ such that
$\sum_{n}e^{-c\Delta_{n}^{\alpha}}<\infty$, where $c$ is the constant
in Lemma \ref{caseI}. For instance, we can take 
\[
\Delta_{n}=\left(\frac{2}{c}\log n\right)^{1/\alpha}.
\]

In the finite symmetric group ${\rm Sym}\left(\partial B_{n}(A)\right)$,
denote by $\Xi_{n}$ the collection of subgroups 
\[
\Xi_{n}:=\bigcup\left\{ L\le{\rm Sym}(U)\times{\rm Sym}(U^{c}):\ {\rm Alt}(U)\times\{id\}\le L\right\} ,
\]
where the union is taken over all subsets $U\subseteq\partial B_{n}(A)$
such that $|U|\ge k_{n}-\Delta_{n}$. Note that the collection $\Xi_{n}$
is invariant under conjugation by ${\rm Sym}\left(\partial B_{n}(A)\right)$. 

Since $\Delta_{n}\ll k_{n}$, if $\Gamma<{\rm Sym}\left(\partial B_{n}(A)\right)$
is such that its giant transitive component $Y_{1}$ has size at least
$k_{n}-\Delta_{n}$ and moreover its projection to ${\rm Sym}(Y_{1})$
contains the alternating group ${\rm Alt}(Y_{1})$, then $\Gamma$
contains ${\rm Alt}(Y_{1})\times\{id\}$, where $id$ is the identity
element of ${\rm Sym}\left(\partial B_{n}(A)\setminus Y_{1}\right)$.
Thus the three lemmas above implies that if $\Gamma_{i}$ is not in
$\Xi_{n}$, then the contribution of $\nu_{\Gamma_{i}}\left(\Theta_{u,v}^{A,n}\right)$
to $\bar{\mu}_{n}^{A}\left(\Theta_{u,v}^{A,n}\right)$ is small. Taking
into account the tree structure, by the Borel-Cantelli lemma we obtain
the following.

\begin{proposition}\label{containRO}

Let $\mu$ be an IRS of $\mathcal{N}_{F}$ where $F$ is transitive
on $D$. Then for any finite complete subtree $A$ and $u,v\in\partial A$
two distinct leaves, there exists a subset $\tilde{\Theta}_{u,v}^{A}\subseteq\Theta_{u,v}^{A}$
with $\mu\left(\tilde{\Theta}_{u,v}^{A}\right)=\mu\left(\Theta_{u,v}^{A}\right)$
such that the following holds. For any $H\in\tilde{\Theta}_{u,v}^{A}$,
there exists a non-empty open set $U\subseteq\partial\mathcal{T}$
such that 
\[
\left[R_{O_{F}^{A}}(U),R_{O_{F}^{A}}(U)\right]\le H.
\]
. 

\end{proposition}

\begin{proof}

Recall that we write $\bar{\mu}_{n}^{A}=\sum_{i=1}^{I_{n}}\lambda_{i}\nu_{\Gamma_{i}}$
for the ergodic decomposition of $\bar{\mu}_{n}^{A}$, where each
$\Gamma_{i}<{\rm Sym}\left(\partial B_{n}(A)\right)$. Denote by $\bar{H}_{n}$
the sub-quotient 
\[
\bar{H}_{n}=\pi_{n}\left(H\cap O_{F}^{A}(n)\right).
\]
Then we have 
\[
\mathbb{P}_{\mu}\left(\bar{H}_{n}\notin\Xi_{n}\mbox{ and }\bar{H}_{n}\in\Theta_{u,v}^{A,n}\right)=\bar{\mu}_{n}^{A}\left(\Xi_{n}^{c}\bigcap\Theta_{u,v}^{A,n}\right)=\sum_{i:\ \Gamma_{i}\notin\Xi_{n}}\lambda_{i}\nu_{\Gamma_{i}}\left(\Theta_{u,v}^{A,n}\right),
\]
where the second equality means in the ergodic decomposition, only
those $\Gamma_{i}$ that are not in $\Xi_{n}$ contribute to the probability
of the event $\Xi_{n}^{c}\bigcap\Theta_{u,v}^{A,n}$. By Lemma \ref{caseI},
\ref{caseII} and \ref{caseIII}, we have that if $\Gamma_{i}\notin\Xi_{n}$,
then there is a constant $C,c>0$ only depending on $q,d$ such that
\[
\nu_{\Gamma_{i}}\left(\Theta_{u,v}^{A,n}\right)\le C\exp\left(-c\Delta_{n}^{\alpha}\right)+Ck_{n}\exp\left(-ck_{n}^{\alpha/3q}\right).
\]
It follows that 
\[
\mathbb{P}_{\mu}\left(\bar{H}_{n}\notin\Xi_{n}\mbox{ and }\bar{H}_{n}\in\Theta_{u,v}^{A,n}\right)\le C\exp\left(-c\Delta_{n}^{\alpha}\right)+Ck_{n}\exp\left(-ck_{n}^{\alpha/6q}\right).
\]
Recall that $\Delta_{n}=\left(\frac{2}{c}\log n\right)^{1/\alpha}$
is chosen so that the sequence $\exp\left(-c\Delta_{n}^{\alpha}\right)$
is summable in $n$. Then we have 
\begin{equation}
\sum_{n=0}^{\infty}\mathbb{P}_{\mu}\left(\bar{H}_{n}\notin\Xi_{n}\mbox{ and }\bar{H}_{n}\in\Theta_{u,v}^{A,n}\right)<\infty.\label{eq:summable}
\end{equation}
Therefore, by the Borel-Cantelli lemma, we have $\mathbb{P}_{\mu}\left(\bar{H}_{n}\in\Xi_{n}^{c}\bigcap\Theta_{u,v}^{A,n}\ \mbox{i.o.}\right)=0$,
where i.o. stands for infinitely often. 

Now consider the event $\Theta_{u,v}^{A}$ as in (\ref{eq:Auv}),
\[
\Theta_{u,v}^{A}=\left\{ H\in{\rm Sub}\left(\mathcal{N}_{F}\right):\ \exists h\in H\cap O_{F}^{A}(0)\mbox{ s.t.}\ v=u\cdot\pi_{0}(h)\right\} .
\]
For $H\in\Theta_{u,v}^{A}$, it follows from the definitions that
$\bar{H}_{n}\in\Theta_{u,v}^{A,n}$ for all $n$. Therefore $\mathbb{P}_{\mu}\left(\bar{H}_{n}\in\Xi_{n}^{c}\bigcap\Theta_{u,v}^{A,n}\ \mbox{i.o.}\right)=0$
implies that the subset $\left\{ H\in\Theta_{u,v}^{A}:\ \bar{H}_{n}\notin\Xi_{n}\mbox{ i.o.}\right\} $
has $\mu$-measure $0$. Denote by $\tilde{\Theta}_{u,v}^{A}$ the
complement of this subset in $\Theta_{u,v}^{A}$, that is, 
\[
\tilde{\Theta}_{u,v}^{A}=\left\{ H\in\Theta_{u,v}^{A}:\ \exists n_{0}(H)\mbox{ s.t. }\bar{H}_{n}\in\Xi_{n}\mbox{ for all }n\ge n_{0}(H)\right\} .
\]
The reasoning above shows that (\ref{eq:summable}) implies $\mu\left(\tilde{\Theta}_{u,v}^{A}\right)=\mu\left(\Theta_{u,v}^{A}\right)$. 

Next we relate back to the tree structure. Take a subgroup $H\in\tilde{\Theta}_{u,v}^{A}$.
For $n\ge n_{0}(H)$, denote by $Y_{H}(n)$ the subset of $\partial B_{n}(A)$
associated with $\bar{H_{n}}$ as in the definition of $\Xi_{n}$,
that is, $Y_{H}(n)$ is the giant transitive component of $\bar{H}_{n}$.
Note that this subset $Y_{H}(n)$ is well-defined and one can recognize
whether a vertex $x$ is in $Y_{H}(n)$ based on the size of the orbit
$x\cdot\bar{H}_{n}$. Recall that $\Delta_{n}=\left(\frac{2}{c}\log n\right)^{1/\alpha}$
is very small compared to the size of $\partial B_{n}(A)$, the latter
being $qd^{n}$. 

\begin{claim}\label{children}

Let $H\in\tilde{\Theta}_{u,v}^{A}$. Then there is a constant $n_{0}=n_{0}(H)$
such that for all $n\ge n_{0}$, a vertex $x\in\partial B_{n}(A)$
is in $Y_{H}(n)$ if and only if all of its children are in $Y_{H}(n+1)$. 

\end{claim}

\begin{proof}[Proof of the Claim]

Let $n_{0}$ be the constant depending on $H$ such that for all $n\ge n_{0}$,
$\bar{H}_{n}\in\Xi_{n}$. 

For the \textquotedbl if\textquotedbl{} direction, denote by $V_{H}(n)$
the set which consists of vertices in $\partial B_{n}(A)$ such that
all of their children are in $Y_{H}(n+1)$. Note that the set $V_{H}(n)$
has cardinality at least $k_{n}-\Delta_{n+1}$. In the next level
$n+1$, ${\rm Alt}(d)\wr_{V_{H}(n)}{\rm Alt}(V_{H}(n))$ is a subgroup
of ${\rm Alt}\left(Y_{H}(n+1)\right)\cap\pi_{n+1}\left(O_{F}^{A}(n)\right)$.
It follows that $\bar{H}_{n}$ is transitive on $V_{H}(n)$. Since
the cardinality of $V_{H}(n)$ is at least $k_{n}-\Delta_{n+1}\gg\Delta_{n}$,
$V_{H}(n)$ has to be contained in the giant transitive component
$Y_{H}(n)$.

For the \textquotedbl only if\textquotedbl{} direction, suppose $x$
is in $Y_{H}(n)$, then its orbit under $\bar{H}_{n}$ has size $\left|Y_{H}(n)\right|$.
Then for any of its children $xi$, $i\in\{0,\ldots,d-1\}$, the orbit
of $xi$ under $\pi_{n+1}(H\cap O_{F}^{A}(n))$ is at least $|Y_{H}(n)|$.
It follows that 
\[
\left|(xi)\cdot\pi_{n+1}\left(H\cap O_{F}^{A}(n+1)\right)\right|\ge|Y_{H}(n)|\ge k_{n}-\Delta_{n}.
\]
Since $k_{n}-\Delta_{n}\gg\Delta_{n+1}$, we conclude that $xi$ must
be in the giant component $Y_{H}(n+1)$. 

\end{proof}

We return to the proof of the proposition. By the Claim above, for
$H\in\tilde{\Theta}_{u,v}^{A}$, we have that if a vertex $x$ is
in $Y_{H}(n)$, where $n\ge n_{0}(H)$, then for the subtree rooted
at $x$, all vertices of distance $\ell$ to $x$ are contained in
$Y_{H}(n+\ell)$. In particular, ${\rm Alt}\left(C_{x}^{\ell}\right)\times\left\{ id_{{\rm Sym}(\partial B_{n+\ell}(A))}\right\} <\pi_{n+\ell}\left(H\cap O_{F}^{A}(n+\ell)\right)$
for all $\ell\ge0$. Since $H$ is a closed subgroup of $\mathcal{N}_{F}$,
we conclude that $H$ contains the derived subgroup $\left[O_{F}^{A}(C_{x}),O_{F}^{A}(C_{x})\right]$. 

\end{proof}

At this moment the proof for Theorem \ref{main} when $F$ is transitive
on $D$ is completed. In the next section we explain the additional
arguments needed to cover the general case of $F$. In particular,
we will make use of Lemma \ref{E2}. 

\section{Proof of Proposition \ref{contain} for general $F$\label{sec:colored}}

We continue to use the notations in Section \ref{sec:Preliminaries}
and \ref{sec:Induced}. Denote by $\left\{ D^{(0)},\ldots,D^{(l)}\right\} $
the $F$-orbits in $D=\{0,1,\ldots,d\}$ and $\ell_{F}$ the labeling
on the vertices of $\mathcal{T}\setminus\{v_{0}\}$ induced by $F$.
Let $A$ be a given complete finite subtree of $\mathcal{T}.$ The
quotient $S_{F}^{A}(n)=O_{F}^{A}(n)/U_{F}^{A}(n)$, which acts on
the set $\partial B_{n}(A)$ preserving the label $\ell_{F}$, is
\[
S_{F}^{A}(n)=\prod_{i=0}^{l}{\rm Sym}\left(D_{n,A}^{(i)}\right),
\]
where $D_{n,A}^{(i)}=\{v\in\partial B_{n}(A):\ \ell_{F}(v)=D^{(i)}\}$.
For each $i\in\{0,\ldots,l\}$, denote by $\vartheta_{i}$ the natural
projection
\[
\vartheta_{i}:\prod_{i=0}^{l}{\rm Sym}\left(D_{n,A}^{(i)}\right)\to{\rm Sym}\left(D_{n,A}^{(i)}\right).
\]
As before, $C_{u}^{n}$ consists of vertices in the subtree rooted
at $u$ of distance $n$ to $u$. The size of $C_{u}^{n}\cap D_{n,A}^{(i)}$
can be calculated as follows. Denote by $M_{F}$ the $(l+1)\times(l+1)$
matrix whose $k$-th row has constant entries $\left(|D^{(k-1)}|,\ldots,|D^{(k-1)}|\right)$.
Then 
\begin{equation}
\left(\begin{array}{c}
\left|C_{u}^{n}\cap D_{n,A}^{(0)}\right|\\
\vdots\\
\left|C_{u}^{n}\cap D_{n,A}^{(l)}\right|
\end{array}\right)=(M_{F}-I)^{n-1}\left(\begin{array}{c}
|D^{(0)}|-\delta_{0}(\ell_{F}(u))\\
\vdots\\
|D^{(l)}|-\delta_{l}(\ell_{F}(u))
\end{array}\right),\label{eq:sizei}
\end{equation}
where $I$ is the $(l+1)\times(l+1)$ identity matrix, $\delta_{j}(\ell_{F}(u))=1$
if $\ell_{F}(u)=D^{(j)}$ and is $0$ otherwise. In particular, asymptotically
we have 
\[
\left|C_{u}^{n}\cap D_{n,A}^{(i)}\right|\sim\frac{1}{l+1}d^{n}.
\]

To proceed, we need a labeled version of Lemma \ref{treematch}, which
follows from the same kind of proof. Recall that the notation $E_{1}\sim_{O_{F}^{A}(0)}E_{2}$
means there exists an element $g\in O_{F}^{A}(0)$ such that $E_{2}=E_{1}\cdot g$. 

\begin{lemma}\label{colormatch}

Let $u,v$ be two distinct vertices in $\partial A$ such that $\ell_{F}(u)=\ell_{F}(v)$
and $i\in\{0,...,l\}.$ Choose a set $E_{1}$ of size $k$ uniformly
random from $C_{u}^{n}\cap D_{n,A}^{(i)}$ and independently choose
a set $E_{2}$ of size $k$ uniformly random from $C_{v}^{n}\cap D_{n,A}^{(i)}$.
Then for any $\delta>0$, there exists constants $c,C>0$ only depending
on $\delta,d$, such that for all $2\le k\le|C_{u}^{n}\cap D_{n,A}^{(i)}|/2$,
\[
\mathbb{P}\left(E_{1}\sim_{O_{F}^{A}(0)}E_{2}\right)\le C\exp\left(-ck^{\frac{d-1}{2d}-\delta}\right).
\]

\end{lemma}

\begin{proof}

Consider recursively down the subtrees rooted at $u,v$. For convenience
of notation, we write $O_{0}=O_{F}^{A}(0)$ through the proof. In
order to have an element in $O_{0}$ which maps $E_{1}$ to $E_{2}$,
it is necessary that there exists a label preserving permutation $\gamma$
such that for each child $uj$ of $u$, $E_{1}\cap C_{uj}^{n-1}\sim_{O_{0}}E_{2}\cap C_{v\gamma(j)}^{n-1}$. 

Take a vertex $y$ in the subtree rooted at $u$ of distance $\ell$
to $u$ and a vertex $z$ in the subtree rooted at $v$ of distance
$\ell$ to $v$ such that $\ell_{F}(y)=\ell_{F}(z)$. Choose a set
$A_{1}$ of size $r$ uniformly from $C_{y}^{n-\ell}\cap D_{n,A}^{(i)}$,
and independently a set $A_{2}$ of size $r$ uniformly from $C_{z}^{n-\ell}\cap D_{n,A}^{(i)}$.
The probability that $A_{1}\sim_{O_{0}}A_{2}$ depends only on the
level $\ell$, the size $r$ and the label $\ell_{F}(y)$. Denote
by $a(\ell,r)$ the maximum of the probability $\mathbb{P}(A_{1}\sim_{O_{0}}A_{2})$
over all $y\in C_{u}^{\ell}$, $z\in C_{v}^{\ell}$. Write for $r\le\frac{1}{2}\min_{z\in C_{u}^{\ell}}\left|C_{z}^{n-\ell}\cap D_{n,A}^{(i)}\right|$,
\[
\tilde{a}(\ell,r)=\sup\left\{ a(\ell,s):\ r\le s\le\frac{1}{2}\max_{z\in C_{u}^{\ell}}\left|C_{z}^{n-\ell}\cap D_{n,A}^{(i)}\right|\right\} .
\]
Since the size of $C_{y}^{n-\ell}\cap D_{n,A}^{(i)}$ satisfies the
equation (\ref{eq:size1}), we have that there is a constant $\lambda_{0}\in(0,1)$
depending on the matrix $M_{F}-I$, such that for any children $yj$
of $y$,
\[
\left|\frac{\left|C_{yj}^{n-\ell-1}\cap D_{n,A}^{(i)}\right|}{\left|C_{y}^{n-\ell}\cap D_{n,A}^{(i)}\right|}-\frac{1}{d}\right|\le\lambda_{0}^{n-\ell-1}.
\]
Take a small constant $\epsilon>0$. Then for level $\ell$ such that
$\lambda_{0}^{n-\ell-1}<\epsilon/2$, the same one step recursion
to children of vertices in level $\ell$ as in the proof of Lemma
\ref{treematch} implies 

\[
\tilde{a}\left(\ell,r\right)\le d\exp\left(-c_{d,\epsilon}r\right)+C_{d,\epsilon}r^{-\frac{d-1}{2}}\tilde{a}\left(\ell+1,\frac{r}{d}-\epsilon r\right)^{d},
\]
 where the constants $c_{d,\epsilon}$ and $C_{d,\epsilon}$ depend
only on $d$ and $\epsilon$. Start with $r$ where $d^{n-\ell}>2r$
and iterate for $s$ steps, where $s$ is such that 
\begin{equation}
C_{d,\epsilon}^{d}r^{-\frac{d-1}{2}}\le\left(\frac{1}{4}\left(\frac{1}{d}-\epsilon\right)^{s}\right)^{d},\label{eq:s-1}
\end{equation}
then summing up the terms we have 
\[
a\left(\ell,r\right)\le C_{1}\exp\left(-c(\epsilon)r\left(1-d\epsilon\right)^{s-1}\right)+2^{-d^{s-1}},
\]
where $C_{1}$ is a constant depending only on $d,\epsilon$. Given
a $\delta>0$, choose $\epsilon$ sufficiently small and $s$ the
largest integer satisfying (\ref{eq:s-1}), we conclude that there
are constants $C,c$ only depending on $\delta,d$ such that 
\begin{equation}
\tilde{a}(\ell,r)\le C\exp\left(-cr^{\frac{d-1}{2d}-\delta}\right).\label{eq:alr-1}
\end{equation}

\end{proof}

Let $\Gamma$ be a subgroup of $S_{F}^{A}(n)$ and $\nu_{\Gamma}$
be the IRS which is the uniform measure on $S_{F}^{A}(n)$-conjugates
of $\Gamma$. With Lemma \ref{colormatch} one can repeat the argument
in Section \ref{sec:transitive} to each component $\vartheta_{i}\left(\Gamma\right)$.
Then we use the conjugacy class size Lemma \ref{E2} to show that
if $\Gamma$ does not contain the product of the large alternating
subgroups then $\nu_{\Gamma}\left(\Theta_{u,v}^{A,n}\right)$ is small.
More precisely, let $\left(\Delta_{n}\right)$ be an increasing sequence
of positive numbers, $\Delta_{n}\ll n$. Denote by $\Pi_{n}^{A}$
be the collection of subgroups of $S_{F}^{A}(n)$ where $L\in\Pi_{n}^{A}$
if and only if there exists a subset $U_{i}\subseteq D_{n,A}^{(i)}$
for each $i\in\{0,\ldots,l\}$ with $\left|U_{i}\right|\ge|D_{n,A}^{(i)}|-\Delta_{n},$
such that $L\le\prod_{i=0}^{l}\left({\rm {\rm Sym}}(U_{i})\times{\rm Sym}\left(D_{n,A}^{(i)}\setminus U_{i}\right)\right)$
and $L\ge\prod_{i=0}^{l}\left({\rm {\rm Alt}}(U_{i})\times\{id\}\right)$. 

As in the previous section, write 

\[
q=|\partial A|,k_{n}=qd^{n}=\left|\partial B_{n}(A)\right|\mbox{ and }\alpha=\frac{d-1}{4d}.
\]
For $\Gamma\notin\Pi_{n}^{A}$, we have the following bound, where
the event $\Theta_{u,v}^{A,n}$ is defined in (\ref{eq:Auvn}). 

\begin{lemma}\label{colorbound}

Suppose that $\Gamma\le S_{F}^{A}(n)$ is such that $\Gamma\notin\Pi_{n}^{A}$.
Then there are constants only depending on $q,d$ such that 
\begin{equation}
\nu_{\Gamma}\left(\Theta_{u,v}^{A,n}\right)\le C\exp\left(-c\Delta_{n}^{\alpha}\right)+Ck_{n}\exp\left(-ck_{n}^{\alpha/6q}\right).\label{eq:tailbound2}
\end{equation}

\end{lemma}

\begin{proof}

Denote by $\vartheta_{i}\left(\Pi_{n}^{A}\right)$ the collection
of subgroups of ${\rm Sym}\left(D_{n,A}^{(i)}\right)$ that can be
obtained as projection of some $L\in\Pi_{n}^{A}$. This collection
is similar to $\Xi_{n}$ considered in the previous section. 

First consider the projection $\vartheta_{i}(\Gamma)$ in the symmetric
group ${\rm Sym}\left(D_{n,A}^{(i)}\right)$, where $i\in\{0,...,l\}.$
In order for $\sigma^{-1}\Gamma\sigma$ to be in $\Theta_{u,v}^{A,n}$,
it is necessary that $\vartheta_{i}\left(\sigma^{-1}\Gamma\sigma\right)$
contains an element $g$ such that $\left(C_{u}^{n}\cap D_{n,A}^{(i)}\right)\cdot g=C_{v}^{n}\cap D_{n,A}^{(i)}$
and there exists $h\in\pi_{n}\left(O_{F}^{A}(0)\right)$ such that
the restriction of $g$ to $C_{u}^{n}\cap D_{n,A}^{(i)}$ coincides
with the restriction of $h$. Apply the same arguments in Section
\ref{sec:transitive} to a random conjugate of $\vartheta_{i}(\Gamma)$
in ${\rm Sym}\left(D_{n,A}^{(i)}\right)$, with Lemma \ref{treematch}
replaced by Lemma \ref{colormatch}, we obtain that if $\Gamma$ is
such that there exists $i\in\{0,\ldots,l\}$ with $\vartheta_{i}(\Gamma)\notin\vartheta_{i}\left(\Pi_{n}^{A}\right)$,
then
\begin{equation}
\nu_{\Gamma}\left(\Theta_{u,v}^{A,n}\right)\le C\exp\left(-c\Delta_{n}^{\alpha}\right)+Ck_{n}\exp\left(-ck_{n}^{\alpha/6q}\right),\label{eq:caseI'}
\end{equation}
where the constants $C,c>0$ depends only on $q,d$. 

Denote by $\Gamma_{i}$ the following normal subgroup of $\Gamma$:
\[
\Gamma_{i}=\{\gamma\in\Gamma:\ \vartheta_{j}(\gamma)=id\mbox{ for all }j\neq i\}.
\]
Next we consider the case where $\vartheta_{j}(\Gamma)\in\vartheta_{j}\left(\Pi_{n}^{A}\right)$
for every $j\in\{0,\ldots,l\}$, but $\Gamma\notin\Pi_{n}^{A}$. Then
there must exist an index $i\in\left\{ 0,\ldots,l\right\} $ and subset
$U_{i}\subseteq D_{n,A}^{(i)}$ such that $\Gamma_{i}$ does not contain
${\rm Alt}(U_{i})$, but 
\begin{equation}
{\rm Alt}(U_{i})\times\{id\}\le\vartheta_{i}(\Gamma)\le{\rm Sym}(U_{i})\times{\rm Sym}\left(D_{n,A}^{(i)}\setminus U_{i}\right).\label{eq:i1}
\end{equation}
Since $\Gamma_{i}$ is normal in $\vartheta_{i}(\Gamma)$, the assumption
that $\Gamma_{i}$ does not contain ${\rm Alt}(U_{i})$ implies that
\begin{equation}
\Gamma_{i}\le\left\{ id_{{\rm Sym}(U_{i})}\right\} \times{\rm Sym}\left(D_{n,A}^{(i)}\setminus U_{i}\right).\label{eq:i2}
\end{equation}
Regard $S_{F}^{A}(n)$ as the product of $L_{1}={\rm Sym}(D_{n,A}^{(i)})$
and $L_{2}=\prod_{j:j\neq i}{\rm Sym}(D_{n,A}^{(j)})$. Now take an
element $g\in O_{F}^{A}(0)$ and apply Lemma \ref{E2} to the IRS
$\nu_{\Gamma}$, then we have 
\begin{equation}
\mathbb{P}_{\nu_{L}}\left(H\ni g\right)\le\mathbb{E}_{\nu_{L}}\left[\frac{1}{\left|{\rm Cl}_{N_{1}}\left(\vartheta_{i}(g)H_{i}\right)\right|}\right],\label{eq:c1}
\end{equation}
where $N_{1}$ is the normalizer of $H_{i}$ in $L_{1}={\rm Sym}(D_{n,A}^{(i)})$,
$H_{i}=\{h\in H:\vartheta_{j}(h)=id\mbox{ for all }j\neq i\}$. 

We now estimate the conjugacy class size which appears in (\ref{eq:c1})
for $g\in\pi_{n}\left(O_{F}^{A}(0)\right)$ with $C_{v}^{n}=C_{u}^{n}\cdot g$.
For $\sigma\in L_{1}$, we have $\left(\sigma^{-1}\Gamma\sigma\right)_{i}\le\left\{ id_{{\rm Sym}(U_{i}\cdot\sigma)}\right\} \times{\rm Sym}\left(D_{n,A}^{(i)}\setminus U_{i}\cdot\sigma\right)$
and the associated normalizer $N_{1}\ge{\rm Sym}\left(U_{i}\cdot\sigma\right)\times\{id\}$.
Note that for any $\sigma\in L_{1}$, 
\[
|\left(U_{i}\cdot\sigma\right)\cap C_{v}^{n}|\ge|C_{v}^{n}\cap D_{n,A}^{(i)}|-\Delta_{n}.
\]
We claim that if $H=\sigma^{-1}\Gamma\sigma$ contains the element
$g$, then the map from the rigid stabilizer of $\left(U_{i}\cdot\sigma\right)\cap C_{v}^{n}$
in $L_{1}$ to conjugacy classes of $\vartheta_{i}(g)H_{i}$, given
by conjugation
\[
\gamma\mapsto\gamma^{-1}\vartheta_{i}(g)\gamma H_{i},
\]
is injective. Indeed the set of partial homeomorphisms $\left\{ h|_{\left(U_{i}\cdot\sigma\right)\cap C_{v}^{n}}:h\in\vartheta_{i}(g)H_{i}\right\} $
consists of a unique element, which is $g|_{\left(U_{i}\cdot\sigma\right)\cap C_{u}^{n}}$.
After the conjugation by $\gamma$ in the rigid stabilizer of $\left(U_{i}\cdot\sigma\right)\cap C_{v}^{n}$,
the set $\left\{ h|_{\left(U_{i}\cdot\sigma\right)\cap C_{v}}:h\in\gamma^{-1}\vartheta_{i}(g)\gamma H_{i}\right\} $
consists of a unique element $g|_{\left(U_{i}\cdot\sigma\right)\cap C_{u}^{n}}\gamma$.
The claim on injectivity follows. We conclude from (\ref{eq:c1})
that 
\[
\mathbb{P}_{\nu_{L}}\left(H\ni g\right)\le\frac{1}{\left(|C_{v}^{n}\cap D_{n,A}^{(i)}|-\Delta_{n}\right)!}.
\]
Take a union bound over $g\in\pi_{n}\left(O_{F}^{A}(0)\right)$ with
$C_{v}^{n}=C_{u}^{n}\cdot g$, we have that 
\begin{equation}
\nu_{L}\left(\Theta_{u,v}^{A,n}\right)\le\frac{\left|\partial A\right|!|F|^{d^{n}}}{\left(|C_{v}^{n}\cap D_{n,A}^{(i)}|-\Delta_{n}\right)!}\le C\exp\left(-cd^{n}\log\frac{d}{C}\right).\label{eq:caseII'}
\end{equation}
The statement follows from combining the two cases (\ref{eq:caseI'})
and (\ref{eq:caseII'}). 

\end{proof} 

Now we conclude the proof of Proposition \ref{contain} stated in
the Introduction. 

\begin{proof}[Proof of Proposition \ref{contain}]

Let $A_{0}$ be a given finite complete tree and $\mu=\mu_{A_{0}}$
an IRS of $O_{F}^{A_{0}}.$ 

We first prove that for every finite complete tree $A$ with $A\supseteq A_{0}$,
and two distinct vertices $u,v\in\partial A$ with $\ell_{F}(u)=\ell_{F}(v)$,
there is a subset $\tilde{\Theta}_{u,v}^{A}$ of $\Theta_{u,v}^{A}$
such that $\mu\left(\Theta_{u,v}^{A}\setminus\tilde{\Theta}_{u,v}^{A}\right)=0$
and for every $H\in\tilde{\Theta}_{u,v}^{A}$, there exists a vertex
in $\mathcal{T}\setminus A$ satisfying $R_{O_{F}^{A}}(C_{x})'<H$. 

Take the sequence $(\Delta_{n})$ to be $\Delta_{n}=\left(\frac{2}{c}\log n\right)^{1/\alpha}$
such that the RHS of (\ref{eq:tailbound2}) is summable in $n$. Then
the same reasoning as in the proof of Proposition \ref{containRO}
shows that Lemma \ref{colorbound} and the Borel-Cantelli Lemma imply
\begin{equation}
\mathbb{P}_{\mu}\left(\pi_{n}\left(H\cap O_{F}^{A}(n)\right)\in\Theta_{u,v}^{A,n}\bigcap\left(\Pi_{n}^{A}\right)^{c}\mbox{ i.o.}\right)=0.\label{eq:i.o.}
\end{equation}
Denote by $\tilde{\Theta}_{u,v}^{A}$ the subset of $\Theta_{u,v}^{A}$
which consists of these $H$ with the property that there exists some
constant $n_{0}=n_{0}(H)$ such that $\pi_{n}\left(H\cap O_{F}^{A}(n)\right)\in\Pi_{n}^{A}$
for all $n$. Then (\ref{eq:i.o.}) implies $\mu\left(\tilde{\Theta}_{u,v}^{A}\right)=\mu\left(\Theta_{u,v}^{A}\right)$. 

For a subgroup $\Gamma\in\Pi_{n}^{A}$, there are well-defined subsets
$U_{i}\subseteq D_{n,A}^{(i)}$, $i\in\left\{ 0,\ldots,l\right\} $,
associated with $\Gamma$, such that for each $i$, the set $U_{i}$
is the gigantic transitive component of $\Gamma$ on label $i$ vertices.
For $H\in\tilde{\Theta}_{u,v}^{A}$, denote by $Y_{H}^{i}(n)$ the
set $U_{i}$ associated with $\pi_{n}\left(H\cap O_{F}^{A}(n)\right)$,
$n\ge n_{0}(H)$. By the same argument as in Claim \ref{children},
a vertex $x$ is in $\cup_{i=0}^{l}Y_{H}^{i}(n)$ if and only if all
of its children are in $\cup_{i=0}^{l}Y_{H}^{i}(n+1)$. Thus if $n$
is a level with $n\ge n_{0}(H)$ and $x\in\cup_{i=0}^{l}Y_{H}^{i}(n)$,
then in the subtree rooted at $x$, for every level $k\in\mathbb{N},$we
have 
\[
\prod_{i=0}^{l}{\rm Alt}\left(C_{x}^{k}\cap D_{n+k,A}^{(i)}\right)\le\pi_{n+k}\left(H\cap O_{F}^{A}(n+k)\right).
\]
Since $H$ is closed, we conclude that $H$ contains the derived subgroup
of $R_{O_{F}^{A}}\left(C_{x}\right)$. 

Note that since $x\in\mathcal{T}\setminus A$ and $A$ is an expansion
of $A_{0}$, we have $R_{O_{F}^{A}}\left(C_{x}\right)=R_{O_{F}^{A_{0}}}\left(C_{x}\right)$.
Finally, since the collection of events $\Theta_{u,v}^{A}$, where
$A$ goes over all expansions of $A_{0}$ and $u,v$ go over all distinct
vertices in $\partial A$ of the same $\ell_{F}$-label, form an open
cover of ${\rm Sub}\left(O_{F}^{A_{0}}\right)\setminus\{id\}$, we
have that the union of $\tilde{\Theta}_{u,v}^{A}$ over all such triples
$A,u,v$ gives the full measure subset in the statement. 

\end{proof}

\bibliographystyle{alpha}
\bibliography{neretin}

\end{document}